\documentclass[12pt,a4paper,leqno]{amsart} 
\usepackage[T1]{fontenc}       
\usepackage{amsfonts}          
\usepackage{amsmath}
\usepackage{amssymb}
\usepackage{overpic}
\usepackage{euscript} 
\usepackage{eurosym}
\usepackage{comment}
\usepackage{mathrsfs} 
\usepackage{ifthen}
\usepackage{esint} 
\usepackage{color}

\long\def\symbolfootnote[#1]#2{\begingroup%
\def\thefootnote{\fnsymbol{footnote}}\footnote[#1]{#2}\endgroup}

{\qed\vspace{5pt}}

\usepackage{amsthm}

\newtheoremstyle{lause}
{5pt}
{5pt}
{\slshape}
{\parindent}
{\bfseries}
{.}
{.5em}
{}

\theoremstyle{lause}

\newtheoremstyle{maaritelma}
{5pt}
{5pt}
{\rmfamily}
{\parindent}
{\bfseries}
{.}
{.5em}
{}

\theoremstyle{maaritelma}
\newtheoremstyle{lause}
{5pt}
{5pt}
{\slshape}
{\parindent}
{\bfseries}
{.}
{.5em}
{}

\theoremstyle{lause}
\newtheorem{theorem}{Theorem}[section]
\newtheorem{lemma}[theorem]{Lemma}

\newtheorem{corollary}[theorem]{Corollary}
\newtheorem{problem}[theorem]{Problem}

\newtheoremstyle{maaritelma}
{5pt}
{5pt}
{\rmfamily}
{\parindent}
{\bfseries}
{.}
{.5em}
{}

\theoremstyle{maaritelma}
\newtheorem{definition}[theorem]{Definition}

\newtheorem{example}[theorem]{Example}
\newtheorem{remark}[theorem]{Remark}

\usepackage{graphicx}
\usepackage{caption}
\numberwithin{equation}{section}

\begin{document}

\thispagestyle{empty}

\begin{center}

{\large{\textbf{Condensers with infinitely many touching Borel
plates and minimum energy problems}}}

\vspace{18pt}

\textbf{Natalia Zorii}

\vspace{18pt}

\textsl{Dedicated to the memory of Professor Bogdan
Bojarski}\vspace{8pt}

\footnotesize{\address{Institute of Mathematics, Academy of Sciences
of Ukraine, Tereshchenkivska~3, 01601,
Kyiv-4, Ukraine\\
natalia.zorii@gmail.com }}

\end{center}

\vspace{12pt}

{\footnotesize{\textbf{Abstract.} Defining a condenser in a locally compact space as a locally finite, countable collection of Borel sets $A_i$, $i\in I$, with the sign $s_i=\pm1$ prescribed such that $A_i\cap A_j=\varnothing$ whenever $s_is_j=-1$, we consider a minimum energy problem with an external field over infinite dimensional vector measures $(\mu^i)_{i\in I}$, where $\mu^i$ is a suitably normalized positive Radon measure carried by $A_i$ and such that $\mu^i\leqslant\xi^i$ for all $i\in I_0$, $I_0\subset I$ and constraints $\xi^i$, $i\in I_0$, being given. If $I_0=\varnothing$, the problem reduces to the (unconstrained) Gauss variational problem, which is in general unsolvable even for a condenser of two closed, oppositely signed plates in $\mathbb R^3$ and the Coulomb kernel. Nevertheless, we provide sufficient conditions for the existence of solutions to the stated problem in its full generality, establish the vague compactness of the solutions, analyze their uniqueness, describe their weighted potentials, and single out their characteristic properties. The strong and the vague convergence of minimizing nets to the minimizers is studied. The phenomena of non-ex\-ist\-ence and non-uniqu\-eness of solutions to the problem are illustrated by examples. The results obtained are new even for the classical kernels on $\mathbb R^n$, $n\geqslant2$, and closed $A_i$, $i\in I$, which is important for applications.}}
\symbolfootnote[0]{\quad 2010 Mathematics Subject Classification:
Primary 31C15.} \symbolfootnote[0]{\quad Key words: infinite
dimensional vector Radon measures, consistent kernels, minimum energy problems,
condensers with infinitely many touching Borel plates, external fields,
constraints.}

\vspace{6pt}

\markboth{\textsl{Natalia Zorii}} {\textsl{Condensers of infinitely many touching Borel plates}}

\section{Introduction}\label{sec:intr}
The interest in minimum energy problems with external fields, initially inspired by Gauss \cite{Gauss} and
further experiencing a new growth due to work
of Frostman \cite{Fr} and Polish and Japanese mathematicians
(Leja et al. and Ohtsuka; see \cite{Leja,O} and the references cited
therein), has been motivated by their direct relations with the
Dirichlet and balayage problems. A new impulse to this part of potential theory (which is often
referred to as \emph{Gauss variational problem\/} or \emph{weighted minimum energy problems\/}) came in the 1980's
when Gonchar and Rakhmanov \cite{GR0,GR1},
Mhaskar and Saff \cite{MaS} applied
logarithmic potentials with external fields in the investigation of
orthogonal polynomials and rational approximations to analytic
functions.

In the present paper we study weighted minimum energy problems in a very
general setting, over infinite dimensional vector measures on a locally compact (Hausdorff) space (l.c.s.) $X$ \cite[Chapter~I, Section~9, n$^\circ$\,7]{B1},
associated with a generalized condenser. To be precise, a \emph{generalized condenser\/} $\mathbf A$ in $X$ is a locally finite, countable collection of \emph{Borel\/} sets
$A_i\subset X$, $i\in I$, termed \emph{plates}, with the sign
$s_i:=\mathrm{sign}\,A_i=\pm1$ prescribed such that $A_i\cap A_j=\varnothing$ whenever $s_is_j=-1$. We
emphasize that, although any two oppositely charged plates of a generalized condenser are disjoint, their closures in $X$ may have points in common.
A generalized condenser $\mathbf A$ is said to be \emph{standard\/} if the
$A_i$, $i\in I$, are closed in $X$. The concept of a standard condenser
with infinitely many (closed) plates has been introduced first in
our earlier study \cite{ZPot2}, while that of a generalized
condenser seems to be new. \emph{Unless explicitly stated otherwise,
when speaking of a condenser, we shall tacitly assume it to be generalized.}

We denote by $\mathfrak M(X)$ the linear space of all real-valued scalar Radon measures on $X$, equipped with the \emph{vague topology}, i.e.,
the (Hausdorff) topology of pointwise convergence on the class $C_0(X)$ of all continuous functions on $X$ with compact support.\footnote{When speaking of a continuous function, we understand that the values are \emph{finite\/} real numbers.} For any set $Q\subset X$,
let $\mathfrak M^+(Q)$ stand for the cone of all \emph{positive\/} $\nu\in\mathfrak M(X)$ \emph{carried by\/} $Q$ (for a definition, see Section~\ref{sec:prel1} below).
These and other notions of the theory of measures and integration on a l.c.s., to be used throughout the
paper, can be found in \cite{E,B2}; see also \cite{F1} for a short
survey.

A vector measure $\boldsymbol\mu=(\mu^i)_{i\in I}$ is said to be \emph{associated\/} with a (generalized) condenser $\mathbf A=(A_i)_{i\in I}$ if $\mu^i\in\mathfrak M^+(A_i)$ for all $i\in I$. Denoting by
$\mathfrak M^+(\mathbf A)$ the class of all those $\boldsymbol\mu$, we thus have\footnote{If $I$ is a singleton, we preserve the normal
fonts instead of the bold ones.}
\[\mathfrak M^+(\mathbf A):=\prod_{i\in I}\,\mathfrak M^+(A_i).\]
The trace of the vague product space topology on $\mathfrak M^+(X)^{\mathrm{Card}\,I}$ on $\mathfrak M^+(\mathbf A)$ is likewise called the \emph{vague topology\/} on $\mathfrak M^+(\mathbf A)$.

For any topological space $Y$, let $\Psi(Y)$ consist of all lower semicontinuous (l.s.c.) functions $\psi: Y\to(-\infty,\infty]$, nonnegative unless $Y$ is compact.

A \emph{kernel\/} on $X$ is defined as a symmetric function $\kappa\in\Psi(X\times X)$. In the present paper we shall be concerned with a \emph{positive definite\/} kernel $\kappa$, which means that the \emph{energy\/} $\kappa(\nu,\nu):=\int\kappa(x,y)\,d(\nu\otimes\nu)(x,y)$
of any (signed) $\nu\in\mathfrak M(X)$ is nonnegative whenever defined. (By definition, $\kappa(\nu,\nu)$ is well defined provided that $\kappa(\nu^+,\nu^+)+\kappa(\nu^-,\nu^-)$ or $\kappa(\nu^+,\nu^-)$ is finite, $\nu^+$ and $\nu^-$ being the positive and negative parts in the Hahn--Jor\-dan decomposition of $\nu$, respectively.) Then the set $\mathcal E_\kappa(X)$ of all
$\nu\in\mathfrak M(X)$ with finite $\kappa(\nu,\nu)$ is a pre-Hil\-bert space with the inner product
\[\langle\mu,\nu\rangle_\kappa:=\kappa(\mu,\nu):=\int\kappa(x,y)\,d(\mu\otimes\nu)(x,y),\quad\mu,\nu\in\mathcal E_\kappa(X),\]
and the seminorm $\|\nu\|_\kappa:=\sqrt{\kappa(\nu,\nu)}$. The topology on $\mathcal E_\kappa(X)$ determined by $\|\cdot\|_\kappa$ is termed \emph{strong}.  A (positive definite) kernel $\kappa$ is said to be \emph{strictly positive definite\/} if the seminorm $\|\cdot\|_\kappa$ is a norm.

In accordance with an electrostatic interpretation of a condenser, assume that the interaction between the components $\mu^i$, $i\in I$, of $\boldsymbol\mu\in\mathfrak M^+(\mathbf A)$ is characterized by the matrix $(s_is_j)_{i,j\in I}$, so that the \emph{energy\/} of $\boldsymbol\mu$ is given by\footnote{An expression $\sum_{i\in
I}c_i$ involving numerical values $c_i$ is meant to be well defined
provided that every summand is so and the sum does not depend on the
order of summation. By the Riemann
series theorem, the sum is finite if and only if the series
converges absolutely. Thus, $\kappa(\boldsymbol\mu,\boldsymbol\mu)$ is
finite provided that $\kappa$ is $(\mu^i\otimes\mu^j)$-int\-egr\-able for all $i,j\in I$ and the series in (\ref{intr1}) converges absolutely.\label{foot-abs}}
\begin{equation}\label{intr1}\kappa(\boldsymbol\mu,\boldsymbol\mu):=\sum_{i,j\in
I}\,s_is_j\kappa(\mu^i,\mu^j).\end{equation}
Let $\mathcal E_\kappa^+(\mathbf A)$ consist of all $\boldsymbol\mu\in\mathfrak M^+(\mathbf A)$ with finite $\kappa(\boldsymbol\mu,\boldsymbol\mu)$ (see footnote~\ref{foot-abs}).

To define admissible measures in the extremal problem we shall be dealing with, fix a numerical vector $\mathbf a=(a_i)_{i\in I}$ with $a_i>0$,
a vector-val\-ued function $\mathbf g=(g_i)_{i\in I}$ with continuous $g_i:X\to(0,\infty)$, and a vector-val\-ued \emph{external field\/} $\mathbf f=(f_i)_{i\in I}$ with universally measurable $f_i:X\to[-\infty,\infty]$.
Let $\mathcal E^+_{\kappa,\mathbf f}(\mathbf A,\mathbf a,\mathbf g)$ consist of all $\boldsymbol\mu\in\mathcal E_\kappa^+(\mathbf A)$ such that $\langle g_i,\mu^i\rangle:=\int g_i\,d\mu^i=a_i$ for all $i\in I$ and $\langle\mathbf f,\boldsymbol\mu\rangle:=\sum_{i\in I}\,\langle f_i,\mu^i\rangle$ is finite (see footnote~\ref{foot-abs}); then so is the \emph{weighted energy\/}
\[G_{\kappa,\mathbf f}(\boldsymbol\mu):=\kappa(\boldsymbol\mu,\boldsymbol\mu)+2\langle\mathbf f,\boldsymbol\mu\rangle, \ \boldsymbol\mu\in\mathcal E^+_{\kappa,\mathbf f}(\mathbf A,\mathbf a,\mathbf g).\]
Fix also $I_0\subset I$ and $\xi^i\in\mathfrak M^+(A_i)$, $i\in I_0$, such that $\langle g_i,\xi^i\rangle>a_i$; these $\xi^i$, $i\in I_0$, will serve as (upper) \emph{constraints\/} acting on positive measures carried by $A_i$, $i\in I_0$. \emph{We shall be concerned with the problem of minimizing the weighted energy\/ $G_{\kappa,\mathbf f}(\boldsymbol\mu)$
over all\/ $\boldsymbol\mu\in\mathcal E^+_{\kappa,\mathbf f}(\mathbf A,\mathbf a,\mathbf g)$ with the additional property that\/ $\mu^i\leqslant\xi^i$ for all\/ $i\in I_0$}.

If $I_0=\varnothing$, the problem reduces to the (unconstrained) Gauss variational problem, which is in general unsolvable even for a standard condenser of two closed, oppositely charged plates in $\mathbb R^n$, $n\geqslant3$, and the Riesz kernels $\kappa_\alpha(x,y):=|x-y|^{\alpha-n}$, $\alpha\in(0,n)$. (Here, $|x-y|$ denotes the Euclidean distance between $x,y\in\mathbb R^n$.) See Theorem~\ref{th:r} below providing necessary and sufficient conditions for the solvability of this problem for $\alpha\in(0,2]$. The phenomenon of unsolvability is illustrated by Example~\ref{ex-thin}.

Nevertheless, we provide sufficient conditions for the existence of solutions to the stated problem in its full generality and establish the vague compactness of the solutions (Theorems~\ref{th1}, \ref{th2}, and \ref{th3}), analyze their uniqueness (Section~\ref{sec:unique}), describe their weighted potentials, and single out their characteristic properties (Theorem~\ref{th:desc} and Corollary~\ref{desc:cont}). The strong and the vague convergence of minimizing nets to the minimizers is also studied (Eq.~(\ref{min-conv-str}) and Corollary~\ref{cor:v:conv}). We discover the phenomenon of non-uniqu\-eness of solutions to the problem, which is illustrated by Example~\ref{ex:nonun}.

\begin{remark}\label{rem:appl}The results obtained are new even for the classical kernels on $\mathbb R^n$, $n\geqslant2$ (in particular, for $-\log|x-y|$ on $\mathbb R^2$), and closed $A_i$, $i\in I$, which is important for applications. While our investigation is focused on theoretical aspects in a very general context, and
possible applications are so far outside the frames of the present paper, it is noteworthy to remark that
minimum energy problems in the constrained and unconstrained settings for the logarithmic kernel and finite dimensional vector measures have been considered for several decades in relation to Hermite--Pad\'{e} approximants \cite{GR,A} and random matrix ensembles \cite{Ku,AK}.\end{remark}

The results of the present paper, mentioned above, are obtained for a condenser with \emph{nearly closed\/} plates, which differ from closed sets in a set of zero inner capacity $c_\kappa(\cdot)$ (Definition~\ref{def:nearly}).\footnote{These closed sets may not form a condenser.} Nevertheless, we develop an efficient approach to the study of energies and potentials of infinite dimensional vector measures for an \emph{arbitrary\/} generalized condenser (Section~\ref{sec:def}), which we intend to use in our further work.

The approach developed is based on the observation that, since $(A_i)_{i\in I}$ is locally finite, the $A_i$, $i\in I$, are Borel, and $A_i\cap A_j=\varnothing$ whenever $s_is_j=-1$, the mapping
\[\mathfrak M^+(\mathbf A)\ni\boldsymbol\mu\mapsto R\boldsymbol\mu:=\sum_{i\in I}\,s_i\mu^i\]
maps $\mathfrak M^+(\mathbf A)$ onto a certain
set of \emph{signed\/} scalar Radon measures on $X$. Furthermore, $\mathcal E^+_\kappa(\mathbf
A)$ becomes a \emph{semimetric\/} space with the semimetric
\begin{equation}\label{vseminorm}
\|\boldsymbol\mu_1-\boldsymbol\mu_2\|_{\mathcal E^+_\kappa(\mathbf
A)}:=\Bigl[\sum_{i,j\in
I}\,s_is_j\kappa(\mu^i_1-\mu^i_2,\mu^j_1-\mu^j_2)\Bigr]^{1/2},
\end{equation}
and $R$ maps $\mathcal E_\kappa^+(\mathbf A)$ \emph{isometrically\/}
onto its (scalar) $R$-im\-age, contained in the pre-Hil\-bert space $\mathcal
E_\kappa(X)$  (see Section~\ref{sec:semimetric}). In view of this isometry, the topology on the semimetric space $\mathcal E^+_\kappa(\mathbf
A)$ is likewise termed \emph{strong}.

Another fact crucial to our approach is a strong completeness result for a certain subspace of $\mathcal E^+_\kappa(\mathbf A)$, where $\mathbf A$ is a standard condenser
(see Theorem~\ref{th:str} below, established in our earlier paper \cite{ZPot2}).
Let $A^+$, resp.\ $A^-$, denote the union of the $A_i$, $i\in I$, with $s_i=+1$, resp.\ $s_i=-1$. Write
\[\mathcal E^+_\kappa(\mathbf A,\leqslant\!\mathbf a,\mathbf g):=\bigl\{\boldsymbol\mu\in\mathcal E^+_\kappa(\mathbf A): \
\langle g_i,\mu^i\rangle\leqslant a_i\text{ \ for all\ }i\in I\bigr\}.\]

\begin{theorem}\label{th:str}Assume the\/ $A_i$, $i\in I$, are closed, $\kappa$ is consistent,\footnote{We refer to \cite{F1,F2} for the concept of \emph{consistency\/} (see also Section~\ref{sec:cons} below).} and
\begin{equation}\label{ser2}\sum_{i\in
I}\,a_ig_{i,\inf}^{-1}:=C<\infty,\text{ \ where \ }g_{i,\inf}:=\inf_{x\in A_i}\,g_i(x).\end{equation}
If, moreover, $\kappa|_{A^+\times A^-}$ is upper bounded, then the following assertions hold.
\begin{itemize}
\item[$\bullet$]
 $\mathcal E^+_\kappa(\mathbf A,\leqslant\!\mathbf a,\mathbf g)$ is complete in the induced strong topology. In more detail, any strong Cauchy net in\/ $\mathcal E^+_\kappa(\mathbf A,\leqslant\!\mathbf a,\mathbf g)$ converges strongly to any of its vague cluster points.
\item[$\bullet$]If, moreover,
$\kappa$ is strictly positive
definite and the\/ $A_i$, $i\in I$, are mutually disjoint, then
the strong topology on\/ $\mathcal E^+_\kappa(\mathbf A,\leqslant\!\mathbf a,\mathbf g)$ is finer than the induced vague topology.
\end{itemize}
\end{theorem}

\subsection{Minimum $\alpha$-Riesz energy problem for a standard condenser}\label{sec:ex}

We next show that the problem in question is in general unsolvable even in the case where $\mathbf A=(A_1,A_2)$ is a  standard condenser in $\mathbb R^n$, $n\geqslant 3$, with $s_1=+1$, $s_2=-1$, $f_1\equiv f_2\equiv 0$, $g_1\equiv g_2\equiv 1$, $a_1=a_2=1$, $I_0=\varnothing$, and $\kappa(x,y):=\kappa_\alpha(x,y):=|x-y|^{\alpha-n}$, $\alpha\in(0,2]$. Under these requirements, the problem can equivalently be rewritten as follows:
\begin{equation}\label{r-min}w_\alpha(\mathbf A):=\inf\,\kappa_\alpha(\mu^1-\mu^2,\mu^1-\mu^2),\end{equation}
where $\mu^i$, $i=1,2$, ranges over the class
\[\mathcal E^+_{\kappa_\alpha}(A_i,1):=\{\nu\in\mathfrak M^+(A_i)\cap\mathcal E_{\kappa_\alpha}(\mathbb R^n):\ \nu(A_i)=1\}.\]
To formulate the corresponding result and to explain in brief the reason for the phenomenon of unsolvability, we first recall the concept of $\alpha$-thin\-ness at infinity.

Throughout Section~\ref{sec:ex}, $F$ denotes a closed set in $\mathbb R^n$, $n\geqslant3$, such that $F^c:=\mathbb R^n\setminus F\ne\varnothing$, and $F^*$ the inverse of $F$ relative to $\{x\in\mathbb R^n: |x-x_0|=1\}$, $x_0\in F^c$ being fixed. Let $\nu^F$ stand for the $\alpha$-Riesz swept measure of $\nu\in\mathfrak M^+(\mathbb R^n)$ onto $F$, determined uniquely by \cite[Theorem~3.6]{FZ}.

\begin{definition}\label{def:thin} $F$ is said to be\/ $\alpha$-\emph{thin at infinity\/} if any of the following four equivalent assertions holds:
\begin{itemize}
\item[\rm (i)] $F^*$ is $\alpha$-thin at $x_0$.
\item[\rm (ii)] Either $F$ is compact, or $x_0$ is an $\alpha$-irregular boundary point of $F^*$.
\item[\rm (iii)] If $F_k$ denotes $F\cap\{x\in\mathbb R^n: q^k\leqslant|x-x_0|<q^{k+1}\}$, where $q\in(1,\infty)$, then
\begin{equation}\label{iii}\sum_{k\in\mathbb N}\,\frac{c_{\kappa_\alpha}(F_k)}{q^{k(n-\alpha)}}<\infty.\end{equation}
\item[\rm (iv)] There exists a connected component $D$ of $F^c$ such that for every $\nu\in\mathfrak M^+(D)$ with $\nu(\mathbb R^n)<\infty$ we have $\nu^F(\mathbb R^n)<\nu(\mathbb R^n)$.
\end{itemize}
\end{definition}

The equivalence of (i) and (ii) is due to \cite[Theorem~VII.13]{Brelo2} or \cite[Theorem~5.10]{L}, that of (ii) and (iii) holds by the Wiener criterion, and that of (iii) and (iv) has been established in \cite[Theorem~3.22]{FZ} (see also earlier papers \cite[Theorem~B]{Z0} and \cite[Theorem~4]{ZR}).

\begin{theorem}\label{th:thin}If\/ $F$ is not\/ $\alpha$-thin at infinity, then\/ $c_{\kappa_\alpha}(F)=\infty$. This cannot be reversed, i.e., there is\/ $F$ with\/ $c_{\kappa_\alpha}(F)=\infty$ that is\/ $\alpha$-thin at infinity.\end{theorem}

\begin{proof} According to \cite[Lemma~5.5]{L}, $c_{\kappa_\alpha}(F)<\infty$$\iff$$\sum_{k\in\mathbb N}\,c_{\kappa_\alpha}(F_k)<\infty$,
$F_k$ being defined in Definition~\ref{def:thin}(iii). When compared with (\ref{iii}), this yields the theorem.
\end{proof}

\begin{remark}There is a false statement in \cite[Chapter~IX, Section~6]{Brelo2} and \cite[p.~216]{Nin} that a closed set $F\subset\mathbb R^n$, $n\geqslant 3$, is $2$-thin at infinity if and only if $c_{\kappa_2}(F)<\infty$. Although for any $\alpha\in(0,2]$, $c_{\kappa_\alpha}(F)<\infty$ implies, indeed, the $\alpha$-thinness of $F$ at infinity, the converse is incorrect (see Theorem~\ref{th:thin} above, as well as \cite[pp.~276--277]{Ca2} pertaining to $\alpha=2$). We emphasize that the $\alpha$-thin\-ness of $F$ at infinity is indeed equivalent to the existence of the $\alpha$-Riesz equilibrium measure $\gamma_F$ on $F$, but treated in an \emph{extended\/} sense where $\gamma_F(F)=\kappa_\alpha(\gamma_F,\gamma_F)=\infty$ is allowed \cite[Theorem~5.1]{L}.\footnote{Mizuta \cite{Mizuta} has shown that the $2$-thin\-ness of a closed planar set at infinity does not necessarily imply the finiteness of the logarithmic capacity, giving thus an answer in the negative to a conjecture by Ninomiya related to the kernel $-\log|x-y|$ on $\mathbb R^2$ \cite[p.~216]{Nin}.}\end{remark}

Returning to problem (\ref{r-min}), we can certainly assume that $c_{\kappa_\alpha}(A_i)>0$, $i=1,2$, for if not, then $w_\alpha(\mathbf A)=+\infty$, and hence the problem makes no sense. There is also no loss of generality in assuming $c_{\kappa_\alpha}(A_1)<\infty$, because if $c_{\kappa_\alpha}(A_i)=\infty$ for $i=1,2$, then $w_\alpha(\mathbf A)=0$; and hence this infimum
cannot be an actual minimum, $\kappa_\alpha$ being strictly positive definite \cite[Theorem~1.15]{L}.

\begin{theorem}[{\rm see \cite[Theorem~5]{ZR}}]\label{th:r} Assume, for simplicity, $A^c_2$ is connected. If, moreover, the Euclidean distance between\/ $A_1$ and\/ $A_2$ is\/ ${}>0$, then problem\/~{\rm(\ref{r-min})} is\/ {\rm(}uniquely\/{\rm)} solvable if and only if either\/ $c_{\kappa_\alpha}(A_2)<\infty$, or\/ $A_2$ is not\/ $\alpha$-thin at infinity.\end{theorem}

It follows that, if $A_2$ is $\alpha$-thin at infinity, but $c_{\kappa_\alpha}(A_2)=\infty$ (such $A_2$ exists by Theorem~\ref{th:thin}), then $w_\alpha(\mathbf A)$ cannot be attained among admissible measures.
The reason for this phenomenon is that, under the quoted assumptions, any minimizing sequence converges strongly and vaguely to a (unique) $\gamma=\gamma^+-\gamma^-$ such that $\gamma^+\in\mathcal E^+_{\kappa_\alpha}(A_1,1)$, while $\gamma^-=(\gamma^+)^{A_2}$ \cite[Eq.~(27)]{ZR}. Since $A_2$ is $\alpha$-thin at infinity, we get $(\gamma^+)^{A_2}(A_2)<1$ by Definition~\ref{def:thin}(iv), and problem (\ref{r-min}) therefore has \emph{no\/} solution.

\begin{figure}[htbp]
\begin{center}
\vspace{-.4in}
\hspace{1in}\includegraphics[width=4in]{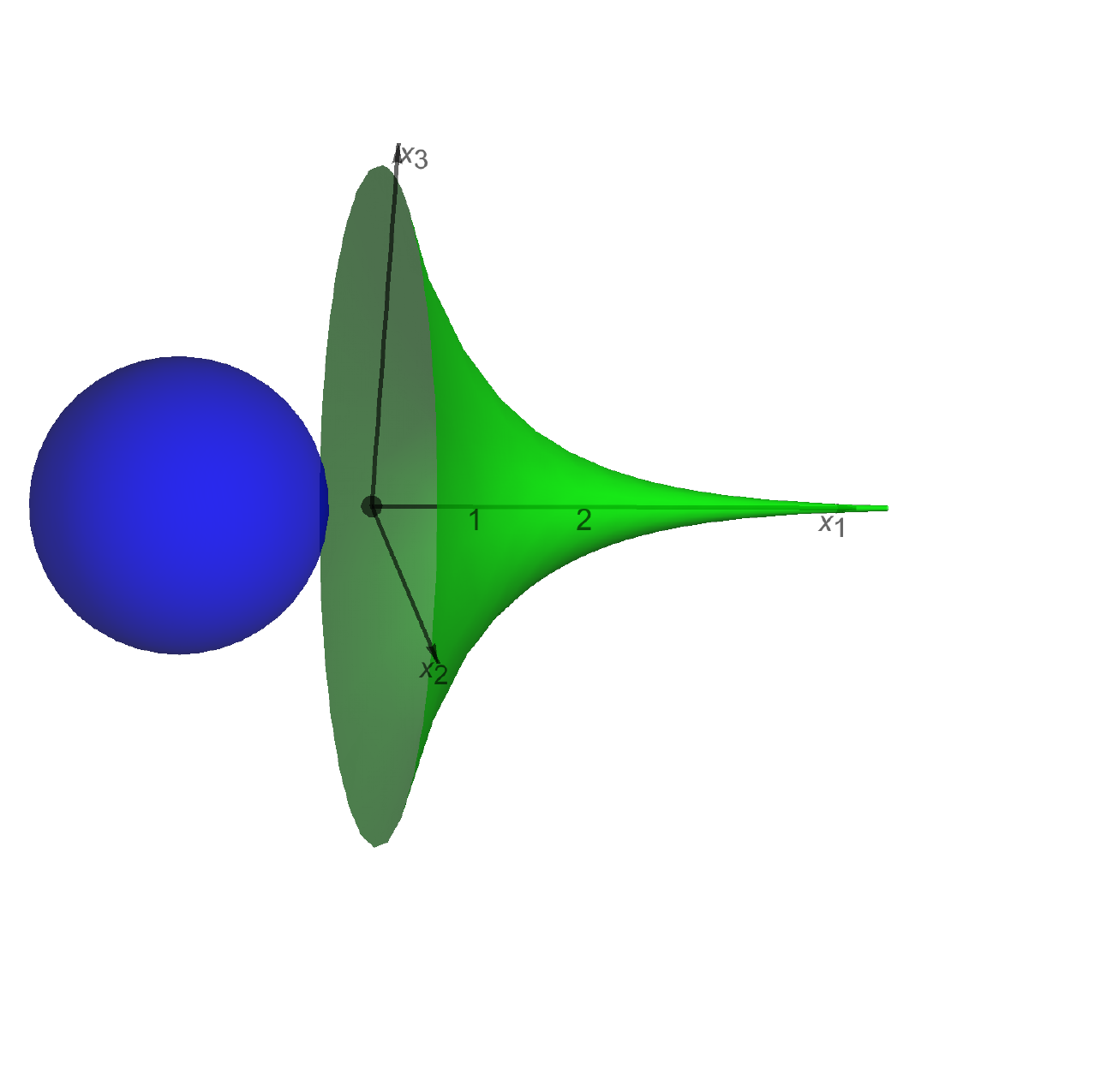}
\vspace{-.8in}
\caption{$\mathbf A=(A_1,A_2)$ in $\mathbb R^3$, where $A_2=\bigl\{0\leqslant x_1<\infty,  \ x_2^2+x_3^2\leqslant\rho^2(x_1)\bigr\}$
with $\rho(x_1)=\exp(-x_1)$ and $A_1$ is a closed ball in
$\mathbb R^3\setminus A_2$.\vspace{-.1in}}
\label{Fig1}
\end{center}
\end{figure}

\begin{example}\label{ex-thin}Let $n=3$ and $\alpha=2$. Define $A_2$ to be a rotation body
\[A_2:=\bigl\{x\in\mathbb R^3: \ 0\leqslant x_1<\infty, \ x_2^2+x_3^2\leqslant\varrho^2(x_1)\bigr\},\]
where $\varrho$ is given by one of the following three formulae:
\begin{align}
\label{c1}\varrho(x_1)&=x_1^{-s}\text{ \ with \ }s\in[0,\infty),\\
\label{c2}\varrho(x_1)&=\exp(-x_1^s)\text{ \ with \ }s\in(0,1],\\
\label{c3}\varrho(x_1)&=\exp(-x_1^s)\text{ \ with \ }s\in(1,\infty),
\end{align}
and let $A_1$ be a closed ball in $\mathbb R^3\setminus A_2$. Then $A_2$ is not $2$-thin at infinity if $\varrho$ is defined by (\ref{c1}), $A_2$ is $2$-thin at infinity but has infinite Newtonian capacity if $\varrho$ is given by (\ref{c2}), and finally $c_{\kappa_2}(A_2)<\infty$ if (\ref{c3}) holds \cite[Example~5.3]{ZPot1}. By Theorem~\ref{th:r}, problem (\ref{r-min}) is therefore solvable for $\mathbf A=(A_1,A_2)$ if $A_2$ is determined either by (\ref{c1}), or by (\ref{c3}), but problem (\ref{r-min}) is \emph{unsolvable\/} if $A_2$ is given by (\ref{c2}) (see Figure~\ref{Fig1}).
\end{example}

\begin{remark}Theorem~\ref{th:r} and Example~\ref{ex-thin} have been illustrated in \cite{OWZ,HWZ} by
means of numerical experiments.\end{remark}

\section{Preliminaries}\label{sec:prel}

\subsection{Measures, vague convergence, capacity}\label{sec:prel1} We shall tacitly use the notation of Section~\ref{sec:intr}. The vague topology on $\mathfrak M(X)$ in general does not possess a countable base, and hence it cannot be described in terms of convergence of sequences. We follow Moore and Smith's theory of convergence, based on the concept of
\emph{nets\/} \cite{MS} (see also \cite[Chapter~2]{K} and \cite[Chapter~0]{E}). However, if $X$ is metrizable and
\emph{countable at infinity}, where the latter means that $X$ can be written as a countable union of compact sets \cite[Chapter~I, Section~9, n$^\circ$\,9]{B1},
then $\mathfrak M(X)$ satisfies the first axiom of
countability \cite[Remark~2.4]{FZ2}, and the use of nets may be
avoided.

\begin{lemma}[{\rm see, e.g., \cite[Section~1.1]{F1}}]\label{lemma:lower}For any\/ $\psi\in\Psi(X)$ the map\/
$\nu\mapsto\langle\psi,\nu\rangle$ is vaguely
l.s.c.\ on\/ $\mathfrak M^+(X)$.\end{lemma}

Let a set $Q\subset X$ and a measure $\nu\in\mathfrak M^+(X)$ be given. If $Q$ is $\nu$-meas\-ur\-able, then the indicator function $1_Q$ of
$Q$ is locally $\nu$-int\-egr\-able, and hence one can consider the \emph{trace\/} (restriction) $\nu|_Q=1_Q\cdot\nu$ of $\nu$ on $Q$
\cite[Section~4.14.7]{E}. As in \cite[Section~4.7.3]{E}, $Q$ is said to be \mbox{$\nu$-$\sigma$}-\emph{fi\-ni\-te} if $Q$ is contained in a countable union of $\nu$-integr\-able open sets.\footnote{This necessarily holds if $X$ is countable at infinity or $\nu$ is \emph{bounded}, i.e., with $\nu(X)<\infty$.}
If $Q$ is open or $\nu$-meas\-ur\-able and \mbox{$\nu$-$\sigma$}-fi\-ni\-te, then $\nu_*(Q)=\nu^*(Q)\in[0,\infty]$, where $\nu_*(Q)$ and $\nu^*(Q)$ denote the \emph{inner\/} and the \emph{outer\/} $\nu$-mea\-sure of $Q$, respectively \cite[Eqs.~(4.7.3), (4.7.4)]{E}; and we write $\nu(Q):=\nu_*(Q)=\nu^*(Q)$.

\begin{lemma}\label{up-int} If\/ $Q$ is\/ $\nu$-meas\-ur\-able and \mbox{$\nu$-$\sigma$}-fi\-ni\-te, then for any nonnegative l.s.c.\ function\/ $\psi$ on\/ $X$ we have\/ $\langle\psi,\nu|_Q\rangle=\langle\psi|_Q,\nu\rangle$.\end{lemma}

\begin{proof} Applying first \cite[Proposition~4.14.1(b)]{E} and \cite[Eq.~(4.14.8)]{E} to $\psi|_Q$ and $\nu$,
and then applying \cite[Proposition~4.14.1(a)]{E} to $\psi$ and $\nu|_Q$, we arrive at our claim.\end{proof}

\begin{theorem}\label{Portmanteau}Let\/ $X$ be metrizable and countable at infinity. If a sequence\/ $\{\nu_k\}_{k\in\mathbb N}\subset\mathfrak M^+(X)$ converges to\/ $\nu$ vaguely, then
for any relatively compact Borel set\/ $Q\subset X$ with\/ $\nu(\partial_XQ)=0$ we have\/ $\nu_k|_Q\to\nu|_Q$ vaguely as\/ $k\to\infty$.\footnote{If $X$ is an open subset of $\mathbb R^n$, $n\geqslant2$, then Theorem~\ref{Portmanteau} is, in fact, \cite[Theorem~0.5$'$]{L}.}
\end{theorem}

\begin{proof} The Portmanteau theorem in the form stated in \cite[Theorem~2.1]{Lin} shows that under the hypotheses of Theorem~\ref{Portmanteau},
\[\lim_{k\to\infty}\,\nu_k(Q)=\nu(Q).\]
Applying now to $X$, $Q$, $\nu_k$ and $\nu$ the same arguments as in \cite[Proof of Theorem~0.5$'$]{L}, the only difference being in using the preceding display in place of \cite[Theorem~0.5]{L}, we establish the theorem.\end{proof}

Let $\mathfrak M^+(Q)$ consist of all $\nu\in\mathfrak M^+(X)$ \emph{carried by\/} $Q$, which means that $Q^c:=X\setminus Q$ is locally $\nu$-neg\-lig\-ible, or equivalently that $Q$ is $\nu$-meas\-ur\-able and $\nu=\nu|_Q$. If $Q^c$ is open or \mbox{$\nu$-$\sigma$}-fi\-ni\-te, then the concept of local $\nu$-neg\-lig\-ib\-ility for $Q^c$ coincides with that of $\nu$-neg\-lig\-ib\-ility; and hence $\nu\in\mathfrak M^+(Q)$ if and only if $\nu^*(Q^c)=0$. Therefore, $\nu$ is carried by a \emph{closed\/} $Q$ if and only if it is supported by $Q$; that is, $S(\nu)\subset Q$, where $S(\nu)$ is the support of~$\nu$.

In all that follows, $\kappa$ is a \emph{positive definite\/} kernel on $X$ (Section~\ref{sec:intr}). For any $Q\subset X$, write $\mathcal E^+_\kappa(Q):=\mathcal E_\kappa(X)\cap\mathfrak M^+(Q)$.
The (\emph{inner\/}) \emph{capacity\/} of $Q$ is given by the formula
\begin{equation}\label{cap-def}c_\kappa(Q):=\Bigl[\,\inf_{\nu\in\mathcal
E_\kappa^+(Q):\ \nu(Q)=1}\,\kappa(\nu,\nu)\Bigr]^{-1}\end{equation}
(see, e.g., \cite{F1,O}). Then $0\leqslant c_\kappa(Q)\leqslant\infty$. (As usual, the
infimum over the empty set is taken to be $+\infty$. We also set
$1\bigl/(+\infty)=0$ and $1\bigl/0=+\infty$.)

A proposition $\mathcal P(x)$ involving a variable point $x\in X$ is said to hold \emph{$c_\kappa$-nearly
everywhere\/} (\emph{$c_\kappa$-n.e.\/}) on $Q$ if
$c_\kappa(N)=0$, where $N$ consists of all $x\in Q$ for which
$\mathcal P(x)$ fails.  We write briefly `n.e.' in
place of `$c_\kappa$-n.e.' if this does not cause any
misunderstanding, and we omit `on $Q$' if $Q=X$.

\begin{lemma}[{\rm see \cite[Lemma~2.3.1]{F1}}]\label{cap0}$c_\kappa(Q)=0$\,$\iff$\,$\mathcal E^+_\kappa(Q)=\{0\}$.\end{lemma}

\subsection{Consistent and perfect kernels}\label{sec:cons} In addition to the strong topology on $\mathcal E_\kappa(X)$, determined by the seminorm $\|\cdot\|_\kappa$ (see Section~\ref{sec:intr}),
it is often useful to consider the so-called \emph{weak\/} topology on $\mathcal E_\kappa(X)$, defined by
means of the seminorms $\nu\mapsto|\kappa(\nu,\mu)|$, where $\mu\in\mathcal E_\kappa(X)$ \cite{F1}. By the Cauchy--Schwarz (Bunyakovski) inequality
\begin{equation*}
|\kappa(\mu,\nu)|\leqslant\|\mu\|_\kappa\cdot\|\nu\|_\kappa,\text{ \ where\ }
\mu,\nu\in\mathcal E_\kappa(X),\end{equation*}
the strong topology on $\mathcal E_\kappa(X)$ is \emph{finer\/} than the weak topology.

Following Fuglede \cite{F1,F2}, we call a (positive definite)
kernel $\kappa$ \emph{consistent\/} if it satisfies either of the following two \emph{equivalent\/} properties:
\begin{itemize}
\item[\rm(C$_1$)] \emph{Every strong Cauchy net in\/
$\mathcal E^+_\kappa(X)$ converges strongly to any of its vague cluster
points\/ {\rm(}whenever these exist\/{\rm)}.}
\item[\rm(C$_2$)] \emph{Every strongly bounded and vaguely convergent net
in\/ $\mathcal E^+_\kappa(X)$ converges weakly to its vague limit.}
\end{itemize}

A kernel $\kappa$ is called \emph{perfect\/} if it is consistent and strictly positive definite \cite[Theorem~3.3]{F1}, or equivalently if the following two conditions are fulfilled  (see \cite[p.~166]{F1}):
\begin{itemize}
\item[\rm(P$_1$)] \emph{$\mathcal E^+_\kappa(X)$ is complete in the induced strong topology.}
\item[\rm(P$_2$)] \emph{The strong topology on\/ $\mathcal E^+_\kappa(X)$ is finer than the induced vague topology on\/ $\mathcal E^+_\kappa(X)$.}
\end{itemize}

\begin{example}\label{rem:clas}On $X=\mathbb R^n$, $n\geqslant3$, the $\alpha$-Riesz kernel $\kappa_\alpha$, $\alpha\in(0,n)$, is strictly positive definite and consistent, and hence altogether perfect \cite{D1}; thus so is the Newtonian kernel $\kappa_2(x,y)=|x-y|^{2-n}$ \cite{Ca}. Recently it has been shown that, if $X=D$ where $D$ is an arbitrary open set in $\mathbb R^n$, $n\geqslant3$, and $G^\alpha_D$, $\alpha\in(0,2]$, is the $\alpha$-Green kernel on $D$ \cite[Chapter~IV, Section~5]{L}, then $\kappa=G^\alpha_D$ is likewise perfect \cite[Theorems~4.9, 4.11]{FZ}. Furthermore, the $2$-Green kernel on a planar $2$-Green\-ian set is strictly positive definite by \cite[Chapter~XIII, Section~7]{Doob} and it is consistent by \cite{E2}, and hence altogether perfect. The logarithmic kernel $-\log\,|x-y|$ on a closed disc in $\mathbb R^2$ of radius ${}<1$ is strictly positive definite, as shown in  \cite[Theorem~1.16]{L}. It is therefore perfect (see \cite{O1}), because it satisfies Frostman's maximum principle by \cite[Theorem~1.6]{L}, and hence is regular by \cite[Eq.~(1.3)]{O}. For analogous results concerning the logarithmic kernel on closed balls of arbitrary finite dimension, see \cite{F1a}.\end{example}

\begin{remark}\label{remma}In contrast to (P$_1$), for a perfect kernel $\kappa$ the whole pre-Hil\-bert space $\mathcal E_\kappa(X)$ is in general strongly \emph{incomplete}, and this is the case even for the $\alpha$-Riesz kernel of order $\alpha\in(1,n)$ on $\mathbb R^n$, $n\geqslant 3$
\cite{Ca}.\end{remark}

\begin{remark}\label{remark}The concept of consistent kernel is an efficient tool in minimum energy problems
over classes of \emph{positive scalar\/} Radon measures with
finite energy. Indeed, if $Q$ is closed, $c_\kappa(Q)\in(0,\infty)$, and $\kappa$ is consistent, then the minimum energy problem in (\ref{cap-def}) has a solution $\lambda$ \cite[Theorem~4.1]{F1}; we shall call this $\lambda$ an (\emph{inner\/}) $\kappa$-\emph{cap\-ac\-itary measure\/} on $Q$. (This $\lambda$ is unique if $\kappa$ is strictly positive definite.)
Later the concept of consistency has been shown to be efficient also in minimum energy problems over classes of vector measures of finite or infinite dimensions associated with a standard condenser \mbox{\cite{ZPot1}--\cite{ZPot3}}. The approach developed in \mbox{\cite{ZPot1}--\cite{ZPot3}} substantially used the assumption of the boundedness of the kernel on the Cartesian product of the oppositely charged plates of a condenser, which made it possible to extend Cartan's proof \cite{Ca} of the strong completeness of the cone $\mathcal E_{\kappa_2}^+(\mathbb R^n)$ of all positive measures on $\mathbb R^n$ with finite Newtonian energy to an arbitrary consistent kernel $\kappa$ on a l.c.s.\ $X$ and suitable classes of (\emph{signed\/}) measures $\mu\in\mathcal E_\kappa(X)$ (compare with Theorem~\ref{th:str} as well as Remark~\ref{remma} above).\end{remark}

\subsection{Nearly closed sets}\label{sec:n.c} The following concept seems to be new (a private communication with Bent Fuglede).

\begin{definition}\label{def:nearly}A set $Q\subset X$ is said
to be \emph{nearly closed}, resp.\ \emph{nearly compact}, if
there exists a closed, resp.\ compact, set $\breve{Q}\subset X$ such
that \[c_\kappa(Q\bigtriangleup\breve{Q})=0,\text{ \ where \ }Q\bigtriangleup\breve{Q}:=(Q\setminus\breve{Q})\cup(\breve{Q}\setminus Q).\]
\end{definition}

\begin{lemma}\label{l:cl-q}For any nearly closed set\/ $Q$, $\mathcal E^+_\kappa(Q)=\mathcal E^+_\kappa(\breve{Q})$.
\end{lemma}

\begin{proof}Note that $Q=[\breve{Q}\cup(Q\setminus\breve{Q})]\setminus(\breve Q\setminus Q)$. Since any set in $X$ with
$c_\kappa(\cdot)=0$ cannot carry any nonzero measure from $\mathcal
E_\kappa^+(X)$ (cf.\ Lemma\/~{\rm\ref{cap0}}), $\mathcal
E^+_\kappa(Q)\subset\mathcal E^+_\kappa(\breve{Q})$. Having reversed $Q$ and $\breve{Q}$, we obtain the converse
inclusion.\end{proof}

\begin{lemma}\label{l:aux2}If a set\/ $Q\subset X$ is nearly closed, then the truncated cone\/
$\{\nu\in\mathcal E^+_\kappa(Q): \|\nu\|_\kappa\leqslant1\}$ is
closed in the induced vague topology.\end{lemma}

\begin{proof}As seen from Lemma~\ref{l:cl-q}, it is enough to establish the lemma for
$\breve{Q}$ in place of $Q$. Since $\mathfrak M^+(\breve{Q})$ is vaguely
closed, $\breve{Q}$ being closed in
$X$, and since the energy $\kappa(\nu,\nu)$ is vaguely l.s.c.\ on
$\mathfrak M^+(X)$ \cite[Lemma~2.2.1(e)]{F1}, the lemma follows.
\end{proof}

\section{Vector measures. Their energies and potentials}\label{sec:def}

\subsection{Vector
measures} Fix a countable set  $I$ of indices $i\in\mathbb N$, and
consider the Cartesian product $\mathfrak
M^+(X)^{\mathrm{Card}\,I}$, equipped with the vague product space
topology. Elements $\boldsymbol\mu=(\mu^i)_{i\in I}$ of $\mathfrak
M^+(X)^{\mathrm{Card}\,I}$, where $\mu^i\in\mathfrak M^+(X)$ for all
$i\in I$, are termed positive
$({\mathrm{Card}\,I})$-dim\-ens\-ion\-al \emph{vector measures\/}
on $X$.

\begin{definition}\label{def:v:b}A set $\mathfrak F\subset\mathfrak
M^+(X)^{\mathrm{Card}\,I}$ is said to be \emph{vaguely bounded\/} if for
every $\varphi\in C_0(X)$,
\[\sup_{\mu\in\mathfrak F}\,|\mu^i(\varphi)|<\infty\text{ \ for all\ }i\in
I.\]\end{definition}

\begin{lemma}\label{lem:vaguecomp} A vaguely bounded set\/
$\mathfrak F\subset\mathfrak M^+(X)^{\mathrm{Card}\,I}$ is vaguely
relatively compact.\end{lemma}

\begin{proof}It is clear from the above definition that for every $i\in
I$, the set
\[\mathfrak F^i:=\bigl\{\mu^i\in\mathfrak M^+(X):  \
\boldsymbol\mu=(\mu^j)_{j\in I}\in\mathfrak F\bigr\}\] is vaguely
bounded, and hence vaguely relatively compact in $\mathfrak M^+(X)$
\cite[Chapter~III, Section~2, Proposition~9]{B2}. Since $\mathfrak
F\subset\prod_{i\in I}\,\mathfrak F^i$, the lemma follows from
Tychonoff's theorem on the product of compact spaces
\cite[Chapter~I, Section~9, Theorem~3]{B1}.\end{proof}

Since $\mathfrak M^+(X)$ is Hausdorff in the vague topology, so is
$\mathfrak M^+(X)^{\mathrm{Card}\,I}$ \cite[Chapter~I, Section~8,
Proposition~7]{B1}, and hence a vague limit of any net
$(\boldsymbol\mu_s)_{s\in S}\subset\mathfrak
M^+(X)^{\mathrm{Card}\,I}$ is \emph{unique\/} if it exists. (Throughout the paper, $S$ denotes an upper directed set of indices~$s$.)

\subsection{Generalized and standard condensers}
Assume that $I=I^+\cup I^-$, where  $I^+\cap I^-=\varnothing$ and $I^-$ is allowed to be empty, and there corresponds to every $i\in I$ a nonempty \emph{Borel\/} set
$A_i\subset X$.

\begin{definition}\label{def:cond} A collection $\mathbf A=(A_i)_{i\in I}$ is termed a \emph{generalized\/ $(I^+,I^-)$-con\-denser\/} (or simply
a \emph{generalized condenser\/}) in $X$ if every compact subset
of $X$ intersects with at most finitely many $A_i$, and moreover
\begin{equation}
A_i\cap A_j=\varnothing\text{ \ for all\ } i\in I^+, \ j\in I^-.
\label{non}
\end{equation}\end{definition}

Writing
\begin{equation}\label{sign}s_i:={\rm sign}\,A_i:=\left\{
\begin{array}{lll} +1&\text{if}&i\in I^+,\\
-1&\text{if}&i\in I^-,\\ \end{array} \right.
\end{equation}
we call $A_i$, $i\in I^+$, and $A_j$, $j\in I^-$,
\emph{positive\/} and \emph{negative plates\/} of the
generalized condenser $\mathbf A$.
Note that any two equally sign\-ed plates may intersect each other
or even coincide. Also note that, although any two oppositely signed
plates are disjoint by (\ref{non}), their closures in $X$ may intersect each other (actually, even in a set of nonzero capacity).\footnote{This remains valid
even in the case where the $A_i$, $i\in I$, are nearly closed.}
Furthermore, it follows from above definition that the sets $A^+:=\bigcup_{i\in I^+}\,A_i$ and $A^-:=\bigcup_{j\in
I^-}\,A_j$ are \emph{Borel\/} and \emph{disjoint}, which will be used substantially in all that follows.

\begin{lemma}\label{l:n-cl}If the\/ $A_i$, $i\in I$, are nearly closed, then so are\/ $A^+$ and\/~$A^-$.\end{lemma}

\begin{proof} With $\breve{A}_i$, $i\in I$, determined by Def\-in\-it\-ion\/~\ref{def:nearly} for $Q=A_i$, write
\begin{equation}\label{brevea}\breve{A}^+:=\bigcup_{i\in I^+}\,\breve{A}_i\text{ \ and \ }
\breve{A}^-:=\bigcup_{j\in I^-}\,\breve{A}_j.\end{equation}
Then
$\breve{A}^\pm$ is closed, for the collection
$(\breve{A}_i)_{i\in I^\pm}$ of (closed) sets $\breve{A}_i$ is locally
finite. Since $c_\kappa(A_i\bigtriangleup\breve{A}_i)=0$ for all $i\in I$, the countable subadditivity of inner capacity on Borel sets
\cite[Lemma~2.3.5]{F1} yields $c_\kappa(A^\pm\bigtriangleup\breve{A}^\pm)=0$.\end{proof}

\begin{definition}A generalized condenser is \emph{standard\/}
if its plates are closed.\end{definition}

Unless explicitly stated otherwise, in all that follows $\mathbf A=(A_i)_{i\in I}$ is a
generalized condenser in $X$. Let $\mathfrak M^+(\mathbf
A)$ consist of all $\boldsymbol\mu=(\mu^i)_{i\in I}\in\mathfrak
M^+(X)^{\mathrm{Card}\,I}$ with $\mu^i\in\mathfrak M^+(A_i)$
for all $i\in I$. In other words, $\mathfrak M^+(\mathbf A)$ stands
for the Cartesian product $\prod_{i\in I}\,\mathfrak M^+(A_i)$,
equipped with the vague topology induced from $\mathfrak
M^+(X)^{\mathrm{Card}\,I}$. Elements of $\mathfrak M^+(\mathbf
A)$ are said to be (vector) measures \emph{associated with\/}~$\mathbf A$.

\begin{lemma}\label{l:v:cl}If a condenser\/ $\mathbf A$ is standard, then\/ $\mathfrak M^+(\mathbf A)$
is vaguely closed in\/ $\mathfrak
M^+(X)^{\mathrm{Card}\,I}$.\end{lemma}

\begin{proof} Noting that the $\mathfrak M^+(A_i)$, $i\in I$, are vaguely closed in $\mathfrak M^+(X)$ ($A_i$ being closed in $X$), we obtain the lemma from \cite[Chapter~I, Section~4, Corollary to Proposition~7]{B1}.\end{proof}

\subsection{Mapping $R:\mathfrak M^+(\mathbf A)\to\mathfrak M(X)$}\label{sect:R}

Since each compact subset of $X$ has points in common with at most
finitely many $A_i$, for every $\varphi\in C_0(X)$ only a finite
number of $\mu^i(\varphi)$, $\boldsymbol\mu=(\mu^i)_{i\in
I}\in\mathfrak M^+(\mathbf A)$ being given, are nonzero. This
implies that there corresponds to every positive vector measure
$\boldsymbol\mu\in\mathfrak M^+(\mathbf A)$ a unique (signed) scalar
Radon measure $R\boldsymbol\mu=R(\boldsymbol\mu)\in\mathfrak M(X)$
such that
\[R\boldsymbol\mu(\varphi)=\sum_{i\in I}\,s_i\mu^i(\varphi)\text{ \ for
all\ }\varphi\in C_0(X),\] $s_i$ being determined by
(\ref{sign}). Since the positive scalar measures $\sum_{i\in
I^+}\mu^i$ and $\sum_{i\in I^-}\mu^i$ are carried by the
nonintersecting Borel sets $A^+$ and $A^-$, respectively, these two
measures are, in fact, the positive and negative parts in the
Hahn--Jor\-dan decomposition of $R\boldsymbol\mu$; i.e.,
$R\boldsymbol\mu=(R\boldsymbol\mu)^+-(R\boldsymbol\mu)^-$, where
\[(R\boldsymbol\mu)^+:=\sum_{i\in I^+}\,\mu^i\text{ \ and \ }
(R\boldsymbol\mu)^-:=\sum_{i\in I^-}\,\mu^i.\] When the dependence
of the mapping $R$ on the (generalized) condenser $\mathbf A$ needs
to be indicated explicitly, we shall write $R_{\mathbf A}$ in place
of~$R$.

The mapping $\mathfrak M^+(\mathbf A)\to\mathfrak M(X)$ thus defined
is in general non-injective; i.e., there exist
$\boldsymbol\mu_1,\boldsymbol\mu_2\in\mathfrak M^+(\mathbf A)$ such
that $\boldsymbol\mu_1\ne\boldsymbol\mu_2$, but
$R\boldsymbol\mu_1=R\boldsymbol\mu_2$. We say that
$\boldsymbol\mu_1,\boldsymbol\mu_2\in\mathfrak M^+(\mathbf A)$ are
\emph{$R$-equ\-iv\-alent\/} if
$R\boldsymbol\mu_1=R\boldsymbol\mu_2$. The relation of
$R$-equ\-iv\-alence on $\mathfrak M^+(\mathbf A)$ implies that of
identity (and hence these two relations on $\mathfrak M^+(\mathbf
A)$ are equivalent) if and only if $A_i\cap A_j=\varnothing$ for all $i\ne j$ (compare
with Lemma~\ref{inj} below).

\begin{lemma}\label{lem:vague'}If a net\/
$(\boldsymbol\mu_s)_{s\in S}\subset\mathfrak M^+(\mathbf A)$
converges vaguely to\/ $\boldsymbol\mu_0\in\mathfrak M^+(\mathbf
A)$, then\/ $R\boldsymbol\mu_s\to R\boldsymbol\mu_0$ vaguely in\/
$\mathfrak M(X)$ as\/ $s$ increases along\/~$S$.\end{lemma}

\begin{proof} This follows directly from the observation that the (compact)
support of any $\varphi\in C_0(X)$ can intersect only finitely
many~$A_i$.\end{proof}

\begin{remark}\label{remT} Lemma~\ref{lem:vague'} cannot in general be
reversed. However, if the $A_i$, $i\in I$, are closed and
mutually disjoint, then for any $(\boldsymbol\mu_s)_{s\in S}$ and $\boldsymbol\mu_0$ in\/ $\mathfrak M^+(\mathbf A)$, the vague convergence of
$(R\boldsymbol\mu_s)_{s\in S}$ to $R\boldsymbol\mu_0$ implies the
vague convergence of $(\boldsymbol\mu_s)_{s\in S}$ to
$\boldsymbol\mu_0$. This is seen by the Tie\-tze--Ury\-sohn
extension theorem \cite[Theorem~0.2.13]{E}.\end{remark}

\subsection{Energies and potentials of vector measures and those of their scalar
$R$-im\-ag\-es}\label{sec:en}
For a (positive definite) kernel $\kappa$ and vector measures $\boldsymbol\mu,\boldsymbol\nu\in\mathfrak M^+(\mathbf A)$,
define the \emph{mutual energy\/}\footnote{With regard to (\ref{vectoren}) and (\ref{vectorpot}), see footnote~\ref{foot-abs}.}
\begin{equation}\label{vectoren}\kappa(\boldsymbol\mu,\boldsymbol\nu):=\sum_{i,j\in
I}\,s_is_j\kappa(\mu^i,\nu^j)\end{equation} and the \emph{vector
potential\/} $\kappa_{\boldsymbol\mu}(\cdot)$ as a
vector-valued function on $X$ with the components
\begin{equation}\label{vectorpot}\kappa^i_{\boldsymbol\mu}(\cdot):=\sum_{j\in
I}\,s_is_j\kappa(\cdot,\mu^j), \ i\in I,\end{equation}
where $\kappa(\cdot,\nu):=\int\kappa(\cdot,y)\,d\nu(y)$ denotes the \emph{potential\/} of $\nu\in\mathfrak M(X)$. For
$\boldsymbol\mu=\boldsymbol\nu$, $\kappa(\boldsymbol\mu,\boldsymbol\nu)$ becomes the
\emph{energy\/} $\kappa(\boldsymbol\mu,\boldsymbol\mu)$ of $\boldsymbol\mu$, cf.~(\ref{intr1}).

Let $\mathcal E_\kappa^+(\mathbf A)$ consist of all $\boldsymbol\mu\in\mathfrak M^+(\mathbf A)$
with finite $\kappa(\boldsymbol\mu,\boldsymbol\mu)$, which means that $\kappa$ is
$(\mu^i\otimes\mu^j)$-int\-egr\-able for all $i,j\in I$ and the series
$\sum_{i,j\in I}\,|\kappa(\mu^i,\mu^j)|$ is convergent (the latter
can be omitted if $X$ is compact, for then $I$ is
finite).

\begin{lemma}\label{enfinite} For\/ $\boldsymbol\mu\in\mathfrak M^+(\mathbf A)$ to
have finite energy, it is sufficient that\/ $\mu^i\in\mathcal E^+_\kappa(A_i)$ for all\/ $i\in I$ and, moreover, $\sum_{i\in I}\,\|\mu^i\|_\kappa<\infty$.\end{lemma}

\begin{proof} In fact, applying the
Cauchy--Schwarz inequality in $\mathcal E_\kappa(X)$, we get
\[\sum_{i,j\in
I}\,|\kappa(\mu^i,\mu^j)|\leqslant\sum_{i,j\in
I}\,\|\mu^i\|_\kappa\|\mu^j\|_\kappa=\Bigl(\sum_{i\in
I}\,\|\mu^i\|_\kappa\Bigr)^2.\vspace{-.2in}\]\end{proof}

The following lemma is crucial for the establishment of relations between energies and potentials of vector
measures $\boldsymbol\mu\in\mathfrak M^+(\mathbf A)$ and those of
their (signed scalar) $R$-images $R\boldsymbol\mu\in\mathfrak M(X)$.

\begin{lemma}\label{integral} Given a generalized\/ $(L^+,L^-)$-condenser\/
$\mathbf B=(B_\ell)_{\ell\in L}$
in a l.c.s.\/ $Y$, consider\/ $\boldsymbol\omega=(\omega^\ell)_{\ell\in
L}\in\mathfrak M^+(\mathbf B)$ and\/ $\psi\in\Psi(Y)$. For\/
$\psi$ to be\/ $|R_{\,\mathbf B}\boldsymbol\omega|$-integrable, it
is necessary and sufficient that\/ $\sum_{\ell\in
L}\,|\langle\psi,\omega^\ell\rangle|<\infty$; and in the affirmative
case,
\begin{equation*}
\langle\psi,R_{\,\mathbf B}\boldsymbol\omega\rangle=\sum_{\ell\in
L}\,s_\ell\langle\psi,\omega^\ell\rangle.\label{lemma11}
\end{equation*}
\end{lemma}

\begin{proof} We can certainly assume $L$ to be infinite, for otherwise the lemma is
obvious. Then $Y$ is
noncompact, and hence $\psi$ is nonnegative. Therefore
\[\langle\psi,(R_{\,\mathbf B}\boldsymbol\omega)^+\rangle\geqslant
\sum_{\ell\in L^+,\ \ell\leqslant N}\,\langle
\psi,\omega^\ell\rangle\text{ \ for all\ }N\in L^+.\] On the other
hand, since $\mathbf B$ is locally finite, the sum of $\omega^\ell$ over all $\ell\in L^+$ that do not
exceed $N$ approaches $(R_{\,\mathbf B}\boldsymbol\omega)^+$ vaguely
as $N$ increases along $L^+$. Hence, by Lemma~\ref{lemma:lower},\footnote{The symbol $\lim_{s\in S}$ denotes a limit as $s$
increases along an upper directed set~$S$.}
\begin{equation*}
\bigl\langle\psi,(R_{\,\mathbf
B}\boldsymbol\omega)^+\bigr\rangle\leqslant\lim_{N\in L^+}\,\sum_{\ell\in
L^+,\ \ell\leqslant
N}\,\langle\psi,\omega^\ell\rangle.\end{equation*} Combining these
two displays and then letting $N$ along $L^+$, we get
\[\bigl\langle\psi,(R_{\,\mathbf
B}\boldsymbol\omega)^+\bigr\rangle=\sum_{\ell\in
L^+}\,\langle\psi,\omega^\ell\rangle.\] Since the same holds for $(R_{\,\mathbf B}\boldsymbol\omega)^-$ and $L^-$, the lemma
follows by subtraction.\end{proof}

\begin{corollary}\label{pot.ener} Fix\/ $\boldsymbol\mu,\boldsymbol\nu\in\mathfrak
M^+(\mathbf A)$ and\/ $x\in X$. Then
\begin{align}
\kappa(R\boldsymbol\mu,R\boldsymbol\nu)&=\sum_{i,j\in
I}\,s_is_j\kappa(\mu^i,\nu^j),\label{mutual}\\
\kappa(x,R\boldsymbol\mu)&=\sum_{i\in
I}\,s_i\kappa(x,\mu^i),\label{poten}
\end{align}
each of the identities being understood in the sense that either of
its sides is finite whenever so is the other and then they coincide.
By\/ {\rm(\ref{vectoren})} and\/ {\rm(\ref{mutual})} with\/
$\boldsymbol\mu=\boldsymbol\nu$,
\begin{equation}\label{iff}\boldsymbol\mu\in\mathcal E_\kappa^+(\mathbf A)\iff
R\boldsymbol\mu\in\mathcal E_\kappa(X).\end{equation}
\end{corollary}

\begin{proof} Relation (\ref{poten}) follows directly from Lemma~\ref{integral}
with $Y=X$, $\mathbf B=\mathbf A$, and
$\psi(\cdot)=\kappa(x,\cdot)$. We next apply Lemma~\ref{integral} to the (generalized) condenser
$\mathbf A\times\mathbf A:=(A_i\times A_j)_{(i,j)\in I\times I}$ in
$X\times X$ with $s_{(i,j)}:=s_is_j$, the function
$\psi:=\kappa\in\Psi(X\times X)$, and the vector measure
$\boldsymbol\mu\otimes\boldsymbol\nu\in\mathfrak M^+(\mathbf
A\times\mathbf A)$, where
$\boldsymbol\mu\otimes\boldsymbol\nu:=(\mu^i\otimes\nu^j)_{(i,j)\in
I\times I}$. Noting that
\[R_{\mathbf A\times\mathbf A}(\boldsymbol\mu\otimes\boldsymbol\nu)=\sum_{i,j\in
I}\,s_is_j\mu^i\otimes\nu^j=(R_{\mathbf A}\boldsymbol\mu)\otimes
(R_{\mathbf A}\boldsymbol\nu),\]
we arrive at (\ref{mutual}).\end{proof}

\begin{corollary}\label{relation}Given\/ $\boldsymbol\mu,\boldsymbol\nu\in\mathcal
E^+_\kappa(\mathbf A)$, we have
\begin{equation}\label{vecen}
\kappa(\boldsymbol\mu,\boldsymbol\nu)=\kappa(R\boldsymbol\mu,R\boldsymbol\nu)=
\sum_{i,j\in I}\,s_is_j\kappa(\mu^i,\nu^j).
\end{equation}
Furthermore, for every\/ $i\in I$, $\kappa_{\boldsymbol\mu}^i(x)$ is
finite n.e.\ and can be written in the form
\begin{equation}\label{Rpot}
\kappa^i_{\boldsymbol\mu}(x)=s_i\kappa(x,R\boldsymbol\mu)=\sum_{j\in
I}\,s_is_j\kappa(x,\mu^j).
\end{equation}
The series in\/ {\rm(\ref{vecen})} as well as in\/ {\rm(\ref{Rpot})}
converges absolutely, the latter being valid~n.e.
\end{corollary}

\begin{proof} It is seen from (\ref{iff}) that
$R\boldsymbol\mu,R\boldsymbol\nu\in\mathcal E_\kappa(X)$; hence,
$\kappa(R\boldsymbol\mu,R\boldsymbol\nu)$ is finite (see, e.g.,
\cite[Lemma~3.1.1]{F1}), which yields (\ref{mutual}) with the
absolutely convergent series on the right-hand side. Compared with
(\ref{vectoren}), this implies (\ref{vecen}). Being the potential of
a (scalar) measure of finite energy relative to the positive definite kernel,
$\kappa(\cdot,R\boldsymbol\mu)$
is finite n.e.\ \cite[p.~164]{F1}. Hence, the series on the right-hand side in
(\ref{poten}) converges absolutely n.e., which together with
(\ref{vectorpot}) establishes~(\ref{Rpot}).\end{proof}

\begin{remark}\label{posen}Since the kernel is positive definite,
(\ref{vecen}) with $\boldsymbol\nu=\boldsymbol\mu$ yields
the \emph{positivity of the energy\/}
$\kappa(\boldsymbol\mu,\boldsymbol\mu)$, which a priori was not
obvious:
\begin{equation}\label{pos-en}\kappa(\boldsymbol\mu,\boldsymbol\mu)
\geqslant0\text{ \ for
all\ }\boldsymbol\mu\in\mathcal E_\kappa^+(\mathbf A).\end{equation}
\end{remark}

\begin{remark}\label{rem-convex}It is clear from the above that
$\mathcal E_\kappa^+(\mathbf A)$ is a
\emph{convex cone}. Indeed, since $\mathfrak M^+(\mathbf A)$ is
so, it is enough to observe that
$R(\beta_1\boldsymbol\mu_1+\beta_2\boldsymbol\mu_2)\in\mathcal
E_\kappa(X)$ for any $\beta_1,\beta_2\in(0,\infty)$ and
$\boldsymbol\mu_1,\boldsymbol\mu_2\in\mathcal E_\kappa^+(\mathbf
A)$. As $R\boldsymbol\mu_1,R\boldsymbol\mu_2\in\mathcal E_\kappa(X)$
by (\ref{iff}), while $R(\beta_1\boldsymbol\mu_1+\beta_2\boldsymbol\mu_2)=\beta_1R\boldsymbol\mu_1+
\beta_2R\boldsymbol\mu_2$,
the convexity of $\mathcal E_\kappa^+(\mathbf A)$ follows from the
linearity of $\mathcal E_\kappa(X)$.
\end{remark}

\subsection{Semimetric space of vector measures with finite
energy}\label{sec:semimetric} We next show that the cone
$\mathcal E_\kappa^+(\mathbf A)$ can be thought of as a semimetric space, isometric
to its (scalar) $R$-im\-age.

\begin{lemma}\label{inj} The relation of\/ $R$-equivalence on\/ $\mathcal
E^+_\kappa(\mathbf A)$ is equivalent to that of identity if and only
if the\/ $A_i$, $i\in I$, are mutually essentially disjoint, i.e.,
\begin{equation}\label{ess:dis}c_\kappa(A_i\cap A_j)=0\text{ \ for all\ }i\ne j.\end{equation}
\end{lemma}

\begin{proof} The sufficiency part is obvious by Lemma~\ref{cap0}. Assume now on the contrary that there are $A_k$ and
$A_\ell$, $k\ne\ell$, with $c_\kappa(A_k\cap A_\ell)>0$; then necessarily $s_ks_\ell=+1$. It follows from
Lemma~\ref{cap0} that there exists a nonzero $\tau\in\mathcal E_\kappa^+(A_k\cap A_\ell)$. Choose
$\boldsymbol\mu=(\mu^i)_{i\in I}\in\mathcal E^+_\kappa(\mathbf A)$
such that $\mu^k|_{A_k\cap A_\ell}-\tau\geqslant0$, and define
$\boldsymbol\mu_m=(\mu_m^i)_{i\in I}\in\mathcal E^+_\kappa(\mathbf
A)$, $m=1,2$, where $\mu_1^k:=\mu^k-\tau$ and $\mu_1^i:=\mu^i$ for
all $i\ne k$, while $\mu_2^\ell:=\mu^\ell+\tau$ and $\mu_2^i:=\mu^i$
for all $i\ne\ell$. Then $R\boldsymbol\mu_1=R\boldsymbol\mu_2$, and
hence $\boldsymbol\mu_1$ and $\boldsymbol\mu_2$ are $R$-equ\-iv\-alent,
but $\boldsymbol\mu_1\ne\boldsymbol\mu_2$.\end{proof}

\begin{theorem}\label{lemma:semimetric} The cone\/
$\mathcal E_\kappa^+(\mathbf A)$ is a semimetric space
with the semimetric\/ $\|\boldsymbol\mu_1-\boldsymbol\mu_2\|_{\mathcal E^+_\kappa(\mathbf
A)}$ defined by\/ {\rm(\ref{vseminorm})},
and this space is isometric to its\/ $R$-im\-age. Assume now\/ $\kappa$ is strictly positive definite. Then\/
$\|\boldsymbol\mu_1-\boldsymbol\mu_2\|_{\mathcal E^+_\kappa(\mathbf
A)}$ becomes a metric if and only if\/ {\rm(\ref{ess:dis})} holds.
\end{theorem}

\begin{proof}Fix any $\boldsymbol\mu_1,\boldsymbol\mu_2\in\mathcal E^+_\kappa(\mathbf A)$.
Applying (\ref{vecen}) to
$\kappa(R\boldsymbol\mu_k,R\boldsymbol\mu_t)$, $k,t=1,2$, and then
combining the equalities obtained, we get
\[\|R\boldsymbol\mu_1-R\boldsymbol\mu_2\|_\kappa^2=\sum_{i,j\in
I}\,s_is_j\kappa(\mu^i_1-\mu^i_2,\mu^j_1-\mu^j_2),\]
where the series converges absolutely. Hence, the sum on the right-hand side in (\ref{vseminorm}) is
${}\geqslant0$. When compared with
(\ref{vseminorm}), the last display yields
\begin{equation}\label{seminorm}
\|\boldsymbol\mu_1-\boldsymbol\mu_2\|_{\mathcal E^+_\kappa(\mathbf A)}=
\|R\boldsymbol\mu_1-R\boldsymbol\mu_2\|_\kappa.\end{equation}
Since $\|\cdot\|_\kappa$ is a seminorm on $\mathcal
E_\kappa(X)$, the former part of the theorem follows.

Assume now $\kappa$ is strictly positive definite. By (\ref{seminorm}), $\|\boldsymbol\mu_1-\boldsymbol\mu_2\|_{\mathcal
E_\kappa^+(\mathbf A)}=0$ if and only if $\boldsymbol\mu_1$ and $\boldsymbol\mu_2$ are $R$-equ\-iv\-al\-ent, while by Lemma~\ref{inj}, the relation of $R$-equ\-iv\-al\-ence on $\mathcal
E_\kappa^+(\mathbf A)$ is equivalent to that of identity if and only if (\ref{ess:dis}) holds.\end{proof}

In view of the isometry between $\mathcal E^+_\kappa(\mathbf A)$ and its $R$-image, contained in the-Hil\-bert space $\mathcal E_\kappa(X)$, the
topology on the semimetric space $\mathcal E^+_\kappa(\mathbf A)$ is likewise termed \emph{strong}.
As usual, $\boldsymbol\mu,\boldsymbol\nu\in\mathcal
E^+_\kappa(\mathbf A)$ are said to be \emph{equivalent\/} in the semimetric space $\mathcal
E^+_\kappa(\mathbf A)$ if
$\|\boldsymbol\mu-\boldsymbol\nu\|_{\mathcal E_\kappa^+(\mathbf
A)}=0$.

\begin{corollary}\label{lemma:potequiv}
$\boldsymbol\mu,\boldsymbol\nu\in\mathcal E^+_\kappa(\mathbf A)$
are equivalent in\/ $\mathcal E^+_\kappa(\mathbf A)$ if and only if
\[\kappa^i_{\boldsymbol\mu}(\cdot)=\kappa_{\boldsymbol\nu}^i(\cdot)\text{ \ n.e.\ for all $i\in I$}.\]
\end{corollary}

\begin{proof}In consequence of (\ref{seminorm}), $\boldsymbol\mu$ and $\boldsymbol\nu$ are equivalent
in $\mathcal E^+_\kappa(\mathbf A)$ if and only if $R\boldsymbol\mu$
and $R\boldsymbol\nu$ are equivalent in $\mathcal E_\kappa(X)$,
which in turn holds if and only if
$\kappa(\cdot,R\boldsymbol\mu)=\kappa(\cdot,R\boldsymbol\nu)$ n.e. \cite[Lemma~3.2.1(a)]{F1}. Combining this with (\ref{Rpot})
establishes the corollary.\end{proof}

Being nonlinear, $\mathcal E_\kappa^+(\mathbf A)$ is not normed. Nevertheless, for any of its elements $\boldsymbol\mu$ it is convenient to write $\|\boldsymbol\mu\|_{\mathcal E_\kappa^+(\mathbf
A)}:=\|\boldsymbol\mu-\boldsymbol0\|_{\mathcal E_\kappa^+(\mathbf
A)}$. Then
\begin{equation}\label{sem-norm}\|\boldsymbol\mu\|^2_{\mathcal E_\kappa^+(\mathbf
A)}=\kappa(\boldsymbol\mu,\boldsymbol\mu)=\kappa(R\boldsymbol\mu,R\boldsymbol\mu)=\|R\boldsymbol\mu\|^2_\kappa.\end{equation}

\section{Minimum energy problems for a generalized condenser}\label{sec-Gauss}

\subsection{Formulation of the problems}\label{sec:form} For a (positive definite) kernel $\kappa$
on $X$ and a (generalized) condenser $\mathbf A=(A_i)_{i\in I}$, we shall consider minimum energy problems
with external fields over certain subclasses of $\mathcal
E^+_\kappa(\mathbf A)$.

Fix a vector-valued \emph{external field\/} $\mathbf f=(f_i)_{i\in I}$, where each
$f_i:X\to[-\infty,\infty]$ is $\mu$-meas\-ur\-able for every
$\mu\in\mathfrak M^+(X)$. The \emph{$\mathbf f$-weighted vector
potential\/} and the \emph{$\mathbf f$-weighted energy\/} of
$\boldsymbol\mu\in\mathcal E_\kappa^+(\mathbf A)$ are defined by
\begin{align}\label{wpot}\mathbf W_{\kappa,\mathbf f}^{\boldsymbol\mu}&:=
\kappa_{\boldsymbol\mu}+\mathbf f,\\
\label{wen}G_{\kappa,\mathbf
f}(\boldsymbol\mu)&:=\kappa(\boldsymbol\mu,\boldsymbol\mu)+
2\langle\mathbf f,\boldsymbol\mu\rangle,\end{align} respectively.
Let $\mathcal E_{\kappa,\mathbf f}^+(\mathbf A)$ consist of all
$\boldsymbol\mu\in\mathcal E_\kappa^+(\mathbf A)$ with finite
$\langle\mathbf f,\boldsymbol\mu\rangle$, which means that every
$f_i$, $i\in I$, is $\mu^i$-int\-egr\-able and the series $\sum_{i\in I}\,\langle
f_i,\mu^i\rangle$ converges absolutely.

Fix a numerical vector $\mathbf a=(a_i)_{i\in I}$ with $a_i>0$,
$i\in I$, a vector-valued function $\mathbf g=(g_i)_{i\in I}$,
where all the $g_i: X\to(0,\infty)$ are (finitely) continuous, and write
\[\mathfrak M^+(\mathbf A,\mathbf a,\mathbf g):=
\bigl\{\boldsymbol\mu\in\mathfrak M^+(\mathbf A): \ \langle
g_i,\mu^i\rangle=a_i\text{ \ for all\ }i\in I\bigr\}.\]
If $\mathcal E^+_{\kappa,\mathbf f}(\mathbf
A,\mathbf a,\mathbf g):=\mathcal E^+_{\kappa,\mathbf f}(\mathbf
A)\cap\mathfrak M^+(\mathbf A,\mathbf a,\mathbf g)$ is nonempty, or
equivalently if
\[G_{\kappa,\mathbf f}(\mathbf A,\mathbf a,\mathbf
g):=\inf_{\boldsymbol\mu\in\mathcal E^+_{\kappa,\mathbf f}(\mathbf
A,\mathbf a,\mathbf g)}\,G_{\kappa,\mathbf
f}(\boldsymbol\mu)<\infty,\] then the following
(\emph{unconstrained\/}) \emph{$\mathbf f$-weighted minimum
energy problem}, also known in the literature as the \emph{Gauss
variational problem\/} (see, e.g., \cite{Gauss,O,ST,ZPot2,ZPot3,FZ3}), makes sense.

\begin{problem}\label{pr1}Does there exist\/
$\boldsymbol\lambda_{\mathbf A}\in\mathcal E^+_{\kappa,\mathbf
f}(\mathbf A,\mathbf a,\mathbf g)$ with
\[G_{\kappa,\mathbf f}(\boldsymbol\lambda_{\mathbf A})=
G_{\kappa,\mathbf f}(\mathbf A,\mathbf a,\mathbf
g)?\]
\end{problem}

Let $\mathfrak C(A_i)$, $i\in I$, consist of all $\xi^i\in\mathfrak
M^+(A_i)$ with $\langle g_i,\xi^i\rangle>a_i$; those $\xi^i$ are
said to be (upper) \emph{constraints\/} for elements of $\mathfrak M^+(A_i,a_i,g_i)$. Given $\xi^i\in\mathfrak C(A_i)$, write
\begin{align*}\mathfrak M^{\xi^i}(A_i,a_i,g_i)&:=\bigl\{\mu^i\in\mathfrak
M^+(A_i,a_i,g_i): \ \mu^i\leqslant\xi^i\bigr\},\\
\mathcal E^{\xi^i}_\kappa(A_i,a_i,g_i)&:=\mathcal
E^+_\kappa(A_i)\cap\mathfrak M^{\xi^i}(A_i,a_i,g_i),\end{align*}
where $\mu^i\leqslant\xi^i$ means that $\xi^i-\mu^i\geqslant0$.

Fix $I_0\subset I$, which might be empty. We generalize
Problem~\ref{pr1} by assuming that for every $i\in I_0$, the
$i$-com\-po\-nents $\mu^i$ of the (new) admissible measures
$\boldsymbol\mu$ are now additionally required not to exceed a fixed
constraint $\xi^i\in\mathfrak C(A_i)$; that is,
$\mu^i\in\mathfrak M^{\xi^i}(A_i,a_i,g_i)$ for all $i\in I_0$.
To be precise, write $\boldsymbol\sigma:=(\sigma^i)_{i\in I}$, where
\[\sigma^i:=\left\{
\begin{array}{cll} \xi^i &  \text{if} & i\in I_0,\\
\infty &  \text{if}  & i\in I\setminus I_0,\\ \end{array} \right.
\]
and define
\begin{align*}\mathfrak M^{\boldsymbol\sigma}(\mathbf A,\mathbf a,\mathbf
g)&:=\prod_{i\in I}\,\mathfrak M^{\sigma^i}(A_i,a_i,g_i),\\
\mathcal E_\kappa^{\boldsymbol\sigma}(\mathbf A,\mathbf a,\mathbf
g)&:=\mathcal E_\kappa^+(\mathbf A)\cap\mathfrak
M^{\boldsymbol\sigma}(\mathbf A,\mathbf a,\mathbf g),\\
\mathcal E^{\boldsymbol\sigma}_{\kappa,\mathbf f}(\mathbf A,\mathbf a,\mathbf
g)&:=\mathcal E^+_{\kappa,\mathbf f}(\mathbf A)\cap\mathfrak
M^{\boldsymbol\sigma}(\mathbf A,\mathbf a,\mathbf g).\end{align*}
Here the formal notation $\mathfrak M^{\infty}(A_i,a_i,g_i)$ means that
\emph{no\/} active upper constraint is imposed on
$\mu^i\in\mathfrak M^+(A_i,a_i,g_i)$, i.e.,
\[\mathfrak M^{\sigma^i}(A_i,a_i,g_i)=\mathfrak M^+(A_i,a_i,g_i)\text{ \  for all\ }i\in
I\setminus I_0.\]

If $\mathcal E^{\boldsymbol\sigma}_{\kappa,\mathbf f}(\mathbf A,\mathbf a,\mathbf
g)$ is nonempty (\emph{which will always be tacitly
required\/}), or equivalently if\,\footnote{See
Lemma~\ref{l:gfinite} below providing sufficient conditions for
(\ref{Gconfin}) to hold. Also note that the (nonempty) class $\mathcal
E^{\boldsymbol\sigma}_{\kappa,\mathbf f}(\mathbf A,\mathbf a,\mathbf
g)$ is \emph{convex\/} (cf.\ Remark~\ref{rem-convex}).}
\begin{equation}\label{Gconfin}
G_{\kappa,\mathbf f}^{\boldsymbol\sigma}(\mathbf A,\mathbf a,\mathbf
g):=\inf_{\boldsymbol\mu\in\mathcal
E^{\boldsymbol\sigma}_{\kappa,\mathbf f}(\mathbf A,\mathbf a,\mathbf
g)}\,G_{\kappa,\mathbf f}(\boldsymbol\mu)<\infty,\end{equation} then
the following generalization of Problem~\ref{pr1} makes
sense.

\begin{problem}\label{pr2}Does there exist\/
$\boldsymbol\lambda^{\boldsymbol\sigma}_{\mathbf A}\in\mathcal
E^{\boldsymbol\sigma}_{\kappa,\mathbf f}(\mathbf A,\mathbf a,\mathbf
g)$ with
\[G_{\kappa,\mathbf f}(\boldsymbol\lambda^{\boldsymbol\sigma}_{\mathbf A})=
G_{\kappa,\mathbf f}^{\boldsymbol\sigma}(\mathbf A,\mathbf a,\mathbf
g)?\]
\end{problem}

Observe that under the (permanent) assumption (\ref{Gconfin}),
Problems~\ref{pr1} also makes sense. In fact, $\mathcal
E^{\boldsymbol\sigma}_{\kappa,\mathbf f}(\mathbf A,\mathbf a,\mathbf
g)\subset\mathcal E^+_{\kappa,\mathbf f}(\mathbf A,\mathbf a,\mathbf
g)$, and hence \begin{equation}\label{in1}G_{\kappa,\mathbf
f}(\mathbf A,\mathbf a,\mathbf g)\leqslant G_{\kappa,\mathbf
f}^{\boldsymbol\sigma}(\mathbf A,\mathbf a,\mathbf
g)<\infty.\end{equation}

Problem~\ref{pr2} reduces to Problem~\ref{pr1} if $I_0=\varnothing$,
while in the case $I_0=I$, Problem~\ref{pr2} is known as the
\emph{constrained Gauss variational problem\/} (see, e.g.,
\cite{DS,Z9,FZ2,FZ3,DFHSZ2}). However, the Gauss variational
problem in either constrained or unconstrained setting has not been
studied yet under the present requirements, where $\mathbf A$ is a
collection of infinitely many touching Borel plates (cf.\ Remark~\ref{rem-W}
below). Finally, in the case where $I_0$ is a nonempty proper subset
of $I$, Problem~\ref{pr2} seems to be newly introduced (even for a
standard condenser), though such problem with mixed upper boundary
conditions looks quite natural and also promising in relation to its
possible applications (cf.\ Remark~\ref{rem:appl}).

\begin{remark}\label{rem-W}The most general study of Problem~\ref{pr1} for a standard condenser of infinitely
many (closed) plates seems to have been
provided in \cite{ZPot2,ZPot3}. It includes, e.g., a complete description of the set of all $\mathbf a=(a_i)_{i\in I}$ for which minimizers $\boldsymbol\lambda_{\mathbf A}$ exist as well as an analysis of their uniqueness, vague compactness, and strong and vague continuity of $\boldsymbol\lambda_{\mathbf A}$ when $\mathbf A$ varies. The weighted potentials of minimizers are described, and their characteristic
properties are singled out.\end{remark}

\subsection{Uniqueness of solutions}\label{sec:unique} We next show that the set of the solutions to
Problem~\ref{pr2} is contained in a certain
equivalence class in $\mathcal E_\kappa^+(\mathbf A)$.

\begin{lemma}\label{l:unique}Any two solutions\/ $\boldsymbol\lambda$ and\/
$\widehat{\boldsymbol\lambda}$ to Problem\/~{\rm\ref{pr2}}
{\rm(}whenever these exist\/{\rm)} are equivalent in\/ $\mathcal
E_\kappa^+(\mathbf A)$, i.e., $\|\boldsymbol\lambda-
\widehat{\boldsymbol\lambda}\|_{\mathcal E_\kappa^+(\mathbf A)}=0$.
\end{lemma}

\begin{proof}This can be shown in a way similar to that in
\cite[Proof of Lemma~5.1]{ZPot2}, based
on the convexity of $\mathcal E^{\boldsymbol\sigma}_{\kappa,\mathbf
f}(\mathbf A,\mathbf a,\mathbf g)$, isometry between $\mathcal
E^+_\kappa(\mathbf A)$ and its (scalar) $R$-im\-age, and the
pre-Hil\-bert structure on $\mathcal E_\kappa(X)$.
Indeed, we get from (\ref{Gconfin}), (\ref{wen}), and (\ref{sem-norm})
\[4G_{\kappa,\mathbf f}^{\boldsymbol{\sigma}}(\mathbf A,\mathbf
a,\mathbf g)\leqslant4G_{\kappa,\mathbf
f}\Bigl(\frac{\boldsymbol\lambda+\widehat{\boldsymbol\lambda}}{2}\Bigr)=
\|R\boldsymbol\lambda+R\widehat{\boldsymbol\lambda}\|_\kappa^2+4\langle\mathbf
f,\boldsymbol\lambda+\widehat{\boldsymbol\lambda}\rangle.\] On the
other hand, applying the parallelogram identity in $\mathcal
E_\kappa(X)$ to $R\boldsymbol\lambda$ and
$R\widehat{\boldsymbol\lambda}$ and then adding and subtracting
$4\langle\mathbf
f,\boldsymbol\lambda+\widehat{\boldsymbol\lambda}\rangle$, we obtain
\[\|R\boldsymbol\lambda-
R\widehat{\boldsymbol\lambda}\|_\kappa^2=
-\|R\boldsymbol\lambda+R\widehat{\boldsymbol\lambda}\|_\kappa^2-4\langle\mathbf
f,\boldsymbol\lambda+
\widehat{\boldsymbol\lambda}\rangle+2G_{\kappa,\mathbf{f}}(\boldsymbol\lambda)+
2G_{\kappa,\mathbf{f}}(\widehat{\boldsymbol\lambda}).\] When
combined with the preceding relation, this yields
\[0\leqslant\|R\boldsymbol\lambda-
R\widehat{\boldsymbol\lambda}\|^2_\kappa\leqslant-
4G_{\kappa,\mathbf{f}}^{\boldsymbol{\sigma}}(\mathbf{A},\mathbf{a},\mathbf{g})+
2G_{\kappa,\mathbf{f}}(\boldsymbol\lambda)+
2G_{\kappa,\mathbf{f}}(\widehat{\boldsymbol\lambda})=0,\]
which in view of (\ref{seminorm}) establishes the lemma.
\end{proof}

\begin{corollary}\label{c:unique}If\/ $\kappa$ is strictly positive definite and the\/
$A_i$, $i\in I$, are mutually essentially disjoint, then a solution
to Problem\/~{\rm\ref{pr2}} is unique.\end{corollary}

\begin{proof}This follows from Lemma~\ref{l:unique} when combined with Theorem~\ref{lemma:semimetric}.
\end{proof}

The following example shows that Corollary~\ref{c:unique} fails in
general if the assumption of mutual essential disjointness of
the $A_i$, $i\in I$, is omitted from its hypotheses.

\begin{example}\label{ex:nonun}Let $X=\mathbb R^n$, $n\geqslant3$, $\kappa=\kappa_2$, $I=I^+=\{1,2\}$, $I_0=\{1\}$, $a_1=a_2=1$, $g_1\equiv g_2\equiv1$,
$f_1\equiv f_2\equiv0$, and let $A_1=A_2=K_0$, $K_0$ being an $(n-1)$-dim\-en\-sional unit sphere. Let $\lambda$ denote the $\kappa_2$-cap\-ac\-it\-ary measure on
$K_0$, which exists (cf.\ Remark~\ref{remark} and Example~\ref{rem:clas}). By symmetry reasons, $\lambda$ coincides up to a normalizing factor with
the $(n-1)$-dim\-en\-sional surface measure $m_{n-1}$ on $K_0$.
Define $\xi^1:=3\lambda$, and consider Problem~\ref{pr2} with these data. It is obvious that
$\boldsymbol\lambda=(\lambda,\lambda)$ is one of its solutions. Choose now compact
disjoint sets $K_k\subset K_0$, $k=1,2$, so that
$m_{n-1}(K_1)=m_{n-1}(K_2)>0$, and define
$\nu=\lambda|_{K_1}-\lambda|_{K_2}$. Then
$\widehat{\boldsymbol\lambda}=(\lambda-\nu,\lambda+\nu)$
is an admissible measure for Problem~\ref{pr2} such that
$R\widehat{\boldsymbol\lambda}=R\boldsymbol\lambda$, and hence
$\kappa(\widehat{\boldsymbol\lambda},\widehat{\boldsymbol\lambda})=
\kappa(\boldsymbol\lambda,\boldsymbol\lambda)$. Thus $\widehat{\boldsymbol\lambda}$ along with $\boldsymbol\lambda$
solves
Problem~\ref{pr2}, though $\widehat{\boldsymbol\lambda}\ne\boldsymbol\lambda$.\end{example}

\section{Permanent assumptions. Supplementary results}\label{sec:perm}

In all that follows we require that either $X$ is countable
at infinity, or
\begin{equation}\label{infg}g_{i,\inf}:=\inf_{x\in
A_i}\,g_i(x)>0\text{ \ for all\ }i\in I.\end{equation}

\begin{lemma}\label{a.e.}Let\/ $\mu^i\in\mathcal E^+_\kappa(A_i)$ be such
that\/ $\langle g_i,\mu^i\rangle=c<\infty$. Then a proposition\/ $\mathcal P(x)$
holds\/ $\mu^i$-almost everywhere\/ {\rm(}$\mu^i$-a.e.\/{\rm)},
provided that it holds n.e.\ on\/~$A_i$.\end{lemma}

\begin{proof}The set $N$ of all $x\in A_i$ for which $\mathcal P(x)$ fails has inner capacity zero, and hence it is locally
$\mu^i$-neg\-lig\-ible \cite[Lemma~2.3.1(iii)]{F1}. Furthermore, $N$ is \mbox{$\mu^i$-$\sigma$}-fi\-ni\-te. This is obvious if $X$ is countable at infinity, while
otherwise (\ref{infg}) holds, and therefore
\begin{equation}\label{gmasses}\mu^i(X)\leqslant
cg_{i,\inf}^{-1}<\infty.\end{equation}
Being locally $\mu^i$-neg\-lig\-ible and \mbox{$\mu^i$-$\sigma$}-fi\-ni\-te, $N$ is $\mu^i$-neg\-lig\-ible as claimed.\end{proof}

When speaking of an external field $\mathbf f=(f_i)_{i\in I}$, we
shall henceforth tacitly assume that Case~I or Case~II holds,
where:
\begin{itemize}
\item[\rm I.] \emph{For every\/ $i\in I$, $f_i\in\Psi(X)$};
\item[\rm II.] \emph{For every\/ $i\in I$, $f_i=s_i\kappa(\cdot,\zeta)$, where a\/
{\rm(}signed\/{\rm)} $\zeta\in\mathcal E_\kappa(X)$ is given}.
\end{itemize}

\begin{lemma}\label{caseii} If Case\/~{\rm II} takes place,
then\/ $\mathcal E_\kappa^+(\mathbf A)=\mathcal E_{\kappa,\mathbf
f}^+(\mathbf A)$; and moreover
\begin{equation}\label{yusss}
G_{\kappa,\mathbf
f}(\boldsymbol\mu)=\|R\boldsymbol\mu+\zeta\|_\kappa^2-\|\zeta\|_\kappa^2\text{
\ for all\ }\boldsymbol\mu\in\mathcal E^+_\kappa(\mathbf
A).\end{equation}
\end{lemma}

\begin{proof}Applying Lemma~\ref{integral} to $\boldsymbol\mu\in\mathcal
E^+_\kappa(\mathbf A)$ and each of
$\kappa(\cdot,\zeta^+),\kappa(\cdot,\zeta^-)\in\Psi(X)$, we get by subtraction
\[\langle\mathbf f,\boldsymbol\mu\rangle=\sum_{i\in
I}\,s_i\int\kappa(x,\zeta)\,d\mu^i(x)=\kappa(\zeta,R\boldsymbol\mu).\]
Substituting this together with (\ref{sem-norm}) into (\ref{wen}), we arrive at~(\ref{yusss}).\end{proof}

\begin{lemma}\label{lemma:minusfinite} In either Case\/~{\rm I} or Case\/~{\rm
II},\footnote{As seen from (\ref{Gconfin}), (\ref{in1}), and (\ref{-infty}), $G_{\kappa,\mathbf f}(\mathbf
A,\mathbf a,\mathbf g)$ and $G^{\boldsymbol\sigma}_{\kappa,\mathbf
f}(\mathbf A,\mathbf a,\mathbf g)$ are both finite.\label{foot:finiteness}}
\begin{equation}\label{-infty}G^{\boldsymbol\sigma}_{\kappa,\mathbf f}(\mathbf
A,\mathbf a,\mathbf g)\geqslant G_{\kappa,\mathbf f}(\mathbf A,\mathbf a,\mathbf
g)>-\infty.\end{equation}
\end{lemma}

\begin{proof} Since in Case~II relation (\ref{-infty}) follows directly
from (\ref{yusss}), it is left to consider Case~I. Assume $X$ is compact, for if not, then $f_i\geqslant0$ for all $i\in I$, and
(\ref{-infty}) holds by (\ref{pos-en}). But then $I$ has
to be finite, while every $f_i$, being l.s.c., is bounded from below on
the (compact) space $X$ by $-c_i$, where $0<c_i<\infty$. In
addition, (\ref{infg}) and, hence, (\ref{gmasses}) with $c=a_i$ both hold for every
$i\in I$ and every $\boldsymbol\mu\in\mathfrak M^+(\mathbf A,\mathbf
a,\mathbf g)$, $g_i$ being a strictly positive continuous
function on $X$. Combining all these together gives
\[-\infty<-c_ia_ig_{i,\inf}^{-1}\leqslant-c_i\sup_{\boldsymbol\mu\in\mathfrak M^+(\mathbf A,\mathbf
a,\mathbf g)}\,\mu^i(X)\leqslant\langle f_i,\mu^i\rangle,\] which in
view of the finiteness of $I$ again leads to~(\ref{-infty}).
\end{proof}

\begin{lemma}\label{l:necess}In view of the\/ {\rm(}permanent\/{\rm)} requirement\/
{\rm(\ref{Gconfin})}, for all\/ $i\in I$
\begin{equation}\label{nec}c_\kappa(A^\circ_i)>0,\text{ \ where\ }
A^\circ_i:=\{x\in A_i: \ |f_i(x)|<\infty\}.
\end{equation}
\end{lemma}

\begin{proof}Let, on the contrary, $c_\kappa(A^\circ_j)=0$ for some $j\in I$. Then, by Lemma~\ref{a.e.},
for every $\boldsymbol\mu\in\mathcal E^+_{\kappa,\mathbf f}(\mathbf
A,\mathbf a,\mathbf g)$ (which exists by  (\ref{in1})) we have $|f_j|=\infty$
$\mu^j$-a.e. But this is impossible because $\mu^j\ne0$
while $f_j$ is $\mu^j$-int\-egr\-able.
\end{proof}

For any $M\in(0,\infty)$ and $i\in I$, write $A_i^M:=\{x\in A^\circ_i: \ |f_i(x)|\leqslant M\}$.

\begin{lemma}\label{l:gfinite}Assume there exist\/ $M,M_1\in(0,\infty)$ that are
independent of\/ $i\in I$ and possessing the properties
\begin{align}&\sum_{i\in I_0}\,\|\xi^i|_{A_i^M}\|_\kappa<\infty,\label{ser1}\\
&\langle g_i,\xi^i|_{A_i^M}\rangle\in(a_i,\infty)\text{ \ for all\ }i\in I_0,\label{Mg}\\
&\inf_{i\in I\setminus
I_0}\,c_\kappa(A_i^{M_1})=:M_3\in(0,\infty].\label{Minf}\end{align}
If, moreover, {\rm(\ref{ser2})} is fulfilled, then\/ {\rm(\ref{Gconfin})} holds.
\end{lemma}

\begin{proof}Fix $\varepsilon\in(0,\infty)$, and for every $i\in I\setminus I_0$ choose $\tau_i\in\mathcal E^+_\kappa(A_i^{M_1})$
of compact support so that $\tau_i(A_i^{M_1})=1$ and $\|\tau_i\|^2_\kappa\leqslant
c_\kappa(A_i^{M_1})^{-1}+\varepsilon$. (Such $\tau_i$
exists since $c_\kappa(A_i^{M_1})$ would be the same if the measures $\nu$ in its definition were required to be of compact support $S(\nu)\subset A_i^{M_1}$, cf.\ \cite[p.~153]{F1}.) In view of (\ref{Minf}), we thus get
\[\|\tau_i\|^2_\kappa\leqslant\varepsilon+M_3^{-1}=:M_4^2\in(0,\infty).\]  Write
\[\widetilde{\nu}_i:=\frac{a_i\nu_i}{\langle
g_i,\nu_i\rangle}\text{ \ for all\ }i\in I,\]
where
\begin{equation*}\nu^i:=\left\{
\begin{array}{lll} \tau_i &  \text{if}  & i\in I\setminus I_0,\\
\xi^i|_{A_i^M} &  \text{if} & i\in I_0.\\
 \end{array} \right.
\end{equation*}
Note that $0<\langle
g_i,\nu_i\rangle<\infty$ for all $i\in I$. In fact, for $i\in I\setminus I_0$ this holds because
\[0<\min_{x\in S(\tau^i)}\,g_i(x)\leqslant\langle
g_i,\nu_i\rangle\leqslant\max_{x\in S(\tau^i)}\,g_i(x)<\infty,\] while for $i\in I_0$ it is valid by (\ref{Mg}).
Also observing that, again by (\ref{Mg}), $\widetilde{\nu}_i\leqslant\xi^i$  for all $i\in I_0$,
we thus get $\widetilde{\nu}_i\in\mathfrak
M^{\sigma^i}(A_i,a_i,g_i)$ for all $i\in I$. Furthermore,
\begin{align*}\sum_{i\in
I}\,\|\widetilde{\nu}_i\|_\kappa&\leqslant\sum_{i\in
I_0}\,\frac{a_i}{\langle
g_i,\xi^i|_{A_i^M}\rangle}\|\xi^i|_{A_i^M}\|_\kappa+M_4\sum_{i\in
I\setminus I_0}\,a_ig_{i,\inf}^{-1}\\
{}&\leqslant\sum_{i\in
I_0}\,\|\xi^i|_{A_i^M}\|_\kappa+M_4\sum_{i\in
I\setminus I_0}\,a_ig_{i,\inf}^{-1}<\infty,\end{align*} where the
second inequality follows from (\ref{Mg}) and the third from
(\ref{ser1}) and (\ref{ser2}). Therefore, by Lemma~\ref{enfinite}, $\boldsymbol{\widetilde{\nu}}:=(\widetilde{\nu}_i)_{i\in
I}\in\mathcal E^{\boldsymbol\sigma}_\kappa(\mathbf A,\mathbf
a,\mathbf g)$. Finally,
\[\sum_{i\in I}\,|\langle f_i,\widetilde{\nu}_i\rangle|\leqslant(M+M_1)\sum_{i\in I}\,
\frac{a_i\nu^i(X)}{\langle
g_i,\nu_i\rangle}\leqslant(M+M_1)\sum_{i\in
I}\,a_ig_{i,\inf}^{-1}<\infty,\] the last inequality being obtained
from (\ref{ser2}). Altogether,
$\boldsymbol{\widetilde{\nu}}\in\mathcal
E^{\boldsymbol\sigma}_{\kappa,\mathbf f}(\mathbf A,\mathbf a,\mathbf
g)$.\end{proof}

If $I$ is finite, Lemma~\ref{l:gfinite} takes the following
much simpler form.

\begin{corollary}\label{cor:gfinite}Let\/ $I$ be finite, and let\/ $c_\kappa(A^\circ_i)>0$ for
all\/ $i\in I\setminus I_0$, $A^\circ_i$ being defined in\/ {\rm(\ref{nec})}. Then\/ {\rm(\ref{Gconfin})} holds
provided that for every\/ $i\in I_0$, \[\langle
g_i,\xi^i|_{A^\circ_i}\rangle>a_i\text{ \ and \ }
\xi^i|_{K_i}\in\mathcal E^+_\kappa(K_i)\text{ \ for every compact\ }
K_i\subset A^\circ_i.\]
\end{corollary}

We drop a proof of Corollary~\ref{cor:gfinite}, since it runs
in a way similar to that for Lemma~\ref{l:gfinite}. Combining this with Lemma~\ref{l:necess} yields the following
assertion.

\begin{corollary}\label{cor:gfinite1}If\/ $I$ is finite and\/ $I_0=\varnothing$, then\/
{\rm(\ref{Gconfin})} and\/ {\rm(\ref{nec})} are
equivalent.\end{corollary}

\begin{definition}\label{def:minnet}A net $(\boldsymbol\mu_s)_{s\in S}\subset
\mathcal E^{\boldsymbol\sigma}_{\kappa,\mathbf f}(\mathbf A,\mathbf
a,\mathbf g)$ is \emph{minimizing\/} in Problem~\ref{pr2}
if
\begin{equation}\label{min-seq:}\lim_{s\in S}\,G_{\kappa,\mathbf f}({\boldsymbol\mu}_s)=
G_{\kappa,\mathbf f}^{\boldsymbol\sigma}(\mathbf A,\mathbf a,\mathbf
g).\end{equation} Let $\mathbb M^{\boldsymbol\sigma}_{\kappa,\mathbf
f}(\mathbf A,\mathbf a,\mathbf g)$ consist of all those
$(\boldsymbol\mu_s)_{s\in S}$; it is nonempty because
of~(\ref{Gconfin}).
\end{definition}

\begin{lemma}\label{l:fund}For any\/ $(\boldsymbol\mu_s)_{s\in S}$ and\/ $(\boldsymbol\nu_t)_{t\in T}$
in\/ $\mathbb M^{\boldsymbol{\sigma}}_{\kappa,\mathbf f}(\mathbf
A,\mathbf a,\mathbf g)$,
\begin{equation}
\lim_{(s,t)\in S\times
T}\,\|\boldsymbol\mu_s-\boldsymbol\nu_t\|_{\mathcal
E^+_\kappa(\mathbf A)}=0, \label{fund:}
\end{equation}
$S\times T$ being the upper directed product\,\footnote{See, e.g.,
\cite[Chapter~2, Section~3]{K}.} of the upper directed sets\/ $S$
and\/~$T$.
\end{lemma}

\begin{proof} In the same manner as in the proof of Lemma~\ref{l:unique} we get
\begin{equation*}0\leqslant\|R\boldsymbol\mu_s-R\boldsymbol\nu_t\|^2_\kappa\leqslant-
4G^{\boldsymbol\sigma}_{\kappa,\mathbf f}(\mathbf A,\mathbf
a,\mathbf g)+ 2G_{\kappa,\mathbf
f}(\boldsymbol\mu_s)+2G_{\kappa,\mathbf
f}(\boldsymbol\nu_t),\end{equation*} which yields (\ref{fund:}) when
combined with (\ref{seminorm}), (\ref{Gconfin}), (\ref{-infty}), and (\ref{min-seq:}).\end{proof}

\begin{corollary}\label{cor:fund}Every\/ $(\boldsymbol\mu_s)_{s\in S}\in
\mathbb M^{\boldsymbol\sigma}_{\kappa,\mathbf f}(\mathbf A,\mathbf
a,\mathbf g)$ is strong Cauchy in\/ $\mathcal E^+_\kappa(\mathbf
A)$.\end{corollary}

\section{Sufficient conditions for the solvability of
Problem~\ref{pr2}}\label{sec:suff}

Throughout Section~\ref{sec:suff} we require the permanent assumptions stated in
Sections~\ref{sec:form} and~\ref{sec:perm}. Furthermore, the $A_i$, $i\in I$, are assumed to be
nearly closed. According to Lemma~\ref{l:n-cl}, then so are both
$A^+$ and $A^-$. Let $\breve{A}^+$ and $\breve{A}^-$ be the (closed) sets defined by~(\ref{brevea}).
We denote by $(\boldsymbol\mu_s)'_{s\in S}$ the cluster set
of any $(\boldsymbol\mu_s)_{s\in S}\subset\mathfrak
M^+(\mathbf A)$ in the vague product space topology on $\mathfrak
M^+(X)^{\mathrm{Card}\,I}$, and
$\mathfrak S^{\boldsymbol\sigma}_{\kappa,\mathbf f}(\mathbf
A,\mathbf a,\mathbf g)$ the class of the solutions to Problem~\ref{pr2}.

\begin{theorem}\label{th1}Let the kernel\/ $\kappa$ be consistent, and let the assumptions
\begin{equation}\label{bound}\sup_{(x,y)\in\breve{A}^+\times\breve{A}^-}\,\kappa(x,y)<\infty\end{equation}
and\/ {\rm(\ref{ser2})} be both fulfilled. Also assume that
\begin{equation}\label{boundd}\langle
g_i,\xi^i\rangle<\infty\text{ \ for all $i\in I_0$},\end{equation}
while for every\/ $i\in
I\setminus I_0$,  the following two conditions are required:
\begin{itemize}\item[$\bullet$] Either\/ $A_i$ is nearly compact, or\/ $c_\kappa(A_i)<\infty$.\footnote{A compact set
$K\subset X$ may be of infinite capacity; $c_\kappa(K)$ is
necessarily finite provided that $\kappa$ is strictly positive
definite \cite{F1}. On the other hand, even for the Newtonian kernel,
closed sets of finite capacity may be noncompact (see, e.g., Example~\ref{ex-thin} above).}
\item[$\bullet$]Either\/ $g_i$ is upper bounded, or there are\/
$r_i\in(1,\infty)$ and\/ $\nu_i\in\mathcal E_\kappa(X)$ such that
\begin{equation}
g_i^{r_i}(x)\leqslant\kappa(x,\nu_i)\text{ \ n.e.~on\ } A_i.
\label{growth}
\end{equation}
\end{itemize}
Then in either Case\/~{\rm I} or Case\/~{\rm II}, $\mathfrak S^{\boldsymbol\sigma}_{\kappa,\mathbf f}(\mathbf
A,\mathbf a,\mathbf g)$ is nonempty, vaguely compact, and given by
\begin{equation}\label{desc-solv}\mathfrak S^{\boldsymbol\sigma}_{\kappa,\mathbf f}(\mathbf
A,\mathbf a,\mathbf g)=\bigcup_{(\boldsymbol\nu_t)_{t\in
T}\in\mathbb M^{\boldsymbol\sigma}_{\kappa,\mathbf f}(\mathbf
A,\mathbf a,\mathbf g)}\,(\boldsymbol\nu_t)_{t\in T}'.\end{equation}
Furthermore, for every\/ $(\boldsymbol\nu_t)_{t\in T}\in\mathbb
M^{\boldsymbol\sigma}_{\kappa,\mathbf f}(\mathbf A,\mathbf a,\mathbf
g)$ and every\/ $\boldsymbol\lambda^{\boldsymbol\sigma}_{\mathbf
A}\in\mathfrak S^{\boldsymbol\sigma}_{\kappa,\mathbf f}(\mathbf
A,\mathbf a,\mathbf g)$,
\begin{equation}\label{min-conv-str}\lim_{t\in
T}\,\|\boldsymbol\nu_t-\boldsymbol\lambda^{\boldsymbol\sigma}_{\mathbf
A}\|_{\mathcal E_\kappa^+(\mathbf A)}=0.\end{equation}
\end{theorem}

\begin{definition}\label{k-infty}Denoting by $\infty_X$ the Alexandroff point of $X$ \cite[Chapter~I, Section~9, n$^\circ$\,8]{B1}, we say that a kernel $\kappa$ \emph{possesses the property\/} $(\infty_X)$ if $\kappa(\cdot,y)\to0$ as $y\to\infty_X$
uniformly over compact sets in~$X$.\end{definition}

The Riesz kernel $\kappa_\alpha$, $\alpha\in(0,n)$, on $\mathbb R^n$, $n\geqslant3$, possesses the property $(\infty_X)$. So does the $2$-Green kernel $G^2_D$ on an open bounded set $D\subset\mathbb R^n$, $n\geqslant2$, provided that $D$ is regular in the sense of the solvability of the classical Dirichlet problem.

\begin{theorem}\label{th2}Assume a l.c.s.\ $X$ is metrizable and countable at infinity,\footnote{Theorem~\ref{th2} remains valid for an \emph{arbitrary\/} l.c.s.\ $X$ if we assume instead that only finitely many $\breve{A}_i$, $i\in I^-$, resp.\ $\breve{A}_i$, $i\in I^+$, can intersect one another (see Remark~\ref{rem:th2}).\label{foot:Part}} while a kernel\/ $\kappa(x,y)$ is continuous
for\/ $x\ne y$ and possesses the property\/ $(\infty_X)$. Let\/ $I^+$, resp.\ $I^-$, be finite, the\/ $A_i$, $i\in I$, be nearly compact, and let
\begin{equation}\label{th2:disj}\breve{A}^+\cap\breve{A}^-=\varnothing.\end{equation}
If, moreover, Case\/~{\rm I} takes place and\/ {\rm(\ref{ser2})} holds, then for any\/ $I_0$ and\/ $\boldsymbol\sigma$ the class\/ $\mathfrak
S^{\boldsymbol\sigma}_{\kappa,\mathbf f}(\mathbf A,\mathbf a,\mathbf
g)$ is nonempty, vaguely compact, and given
by\/~{\rm(\ref{desc-solv})}.
\end{theorem}

\begin{remark}If compared with Theorem~\ref{th1}, in Theorem~\ref{th2} the kernel $\kappa$ is not required to be consistent.
However, if under the hypotheses of
Theorem~\ref{th2}, the consistency of $\kappa$ takes place, then Theorem~\ref{th2} becomes valid in both Cases~I
and~II; and moreover, then (\ref{min-conv-str}) also holds.\end{remark}

Recall that a kernel $\kappa$ is said to satisfy the \emph{continuity principle\/} (or to be \emph{regular\/}) if for any $\mu\in\mathfrak M^+(X)$ with compact support $S(\mu)$, $\kappa(\cdot,\mu)$ is continuous on $X$ whenever its restriction to $S(\mu)$ is continuous. The Riesz kernel $\kappa_\alpha$, $\alpha\in(0,n)$, on $\mathbb R^n$, $n\geqslant3$, is regular \cite[Theorem~1.7]{L}. So is the $\alpha$-Green kernel $G_D^\alpha$, $\alpha\in(0,2]$, on an open set $D\subset\mathbb R^n$, $n\geqslant 3$ \cite[Corollary~4.8]{FZ}, as well as the logarithmic kernel on $\mathbb R^2$, the latter being seen by combining \cite[Theorem~1.6]{L} and \cite[Eq.~(1.3)]{O}.

\begin{theorem}\label{th3}Assume\/ $I$ is finite and the\/ $A_i$, $i\in I$,
are nearly compact. Let
the kernel\/ $\kappa$ be regular, and let the\/ $\kappa(\cdot,\xi^i)|_{\breve{A}_i}$,
$i\in I_0$, as well as the\/ $\kappa|_{\breve{A}_i\times\breve{A}_i}$, $i\in I\setminus I_0$, be continuous. Then in either Case\/~{\rm I} or
Case\/~{\rm II} and for any\/ $\mathbf a$ and\/ $\mathbf g$, the
conclusion of Theorem\/~{\rm\ref{th1}} remains valid.\footnote{Theorem~\ref{th3} is applicable to the classical kernels only provided that $I_0=I$.}\end{theorem}

\begin{remark} In contrast to Theorem~\ref{th2}, in Theorem~\ref{th3} the sets $\breve{A}^+$ and $\breve{A}^-$ may have points in common. But then necessarily
$c_\kappa(\breve{A}^+\cap\breve{A}^-)=0$, and hence $\breve{A}^+\cap\breve{A}^-$ cannot carry any nonzero $\nu\in\mathcal E_\kappa(X)$ (see Lemma~\ref{cap0}).
\end{remark}

\begin{corollary}\label{cor:v:conv}Under the hypotheses of any of Theorems\/~{\rm\ref{th1}}, {\rm\ref{th2}},
or\/ {\rm\ref{th3}}, if moreover\/ $\kappa$ is
strictly positive definite, while the\/ $A_i$, $i\in I$, are
mutually essentially disjoint, then\/ $\mathfrak S^{\boldsymbol\sigma}_{\kappa,\mathbf f}(\mathbf
A,\mathbf a,\mathbf g)$ reduces to a unique element\/
$\boldsymbol\lambda^{\boldsymbol\sigma}_{\mathbf A}$, and every\/
$(\boldsymbol\nu_t)_{t\in T}\in\mathbb
M^{\boldsymbol\sigma}_{\kappa,\mathbf f}(\mathbf A,\mathbf a,\mathbf
g)$ converges to this\/
$\boldsymbol\lambda^{\boldsymbol\sigma}_{\mathbf A}$
vaguely.\end{corollary}

\section{Proofs of Theorems~\ref{th1}, \ref{th2}, and \ref{th3} and Corollary~\ref{cor:v:conv}}

\subsection{Auxiliary results}\label{sec:lemmas} Throughout Section~\ref{sec:lemmas}, the $A_i$, $i\in I$,
are assumed to be nearly closed. Write
\[\mathcal E_\kappa^{\boldsymbol\sigma}(\mathbf A,\leqslant\!\mathbf
a,\mathbf g):=\bigl\{\boldsymbol\nu\in\mathcal E^+_\kappa(\mathbf
A): \ \nu^i\leqslant\sigma^i, \ \langle g_i,\nu^i\rangle\leqslant
a_i\text{ \ for all\ }i\in I\bigr\}.\]

\begin{lemma}\label{v-r-c}If\/ {\rm(\ref{ser2})} and\/ {\rm(\ref{bound})} both
hold, then the vague cluster set\/ $(\boldsymbol\mu_s)'_{s\in S}$ of any\/ $(\boldsymbol\mu_s)_{s\in S}\in\mathbb
M^{\boldsymbol\sigma}_{\kappa,\mathbf f}(\mathbf A,\mathbf a,\mathbf
g)$ is nonempty, and moreover
\begin{equation}\label{inclusion1}(\boldsymbol\mu_s)'_{s\in S}\subset
\mathcal E_\kappa^{\boldsymbol\sigma}(\mathbf A, \leqslant\!\mathbf
a,\mathbf g).\end{equation}
\end{lemma}

\begin{proof}Fix a net $(\boldsymbol\mu_s)_{s\in S}\in\mathbb
M^{\boldsymbol\sigma}_{\kappa,\mathbf f}(\mathbf A,\mathbf a,\mathbf
g)$. It is strong Cauchy in the
semimetric space $\mathcal E^+_\kappa(\mathbf A)$ by
Corollary~\ref{cor:fund}, and hence strongly bounded, i.e.,
\begin{equation}\label{net:bound}\sup_{s\in S}\,\|\boldsymbol\mu_s\|^2_{\mathcal E^+_\kappa(\mathbf
A)}=\sup_{s\in
S}\,\|R\boldsymbol\mu_s\|^2_\kappa<\infty,\end{equation} the
equality being valid by (\ref{sem-norm}). Furthermore, it follows from (\ref{ser2}) that (\ref{infg}) and, hence,
(\ref{gmasses}) (with $a_i$ and $\mu_s^i$ in place of $c$ and $\mu^i$) both hold. Thus,
\begin{equation}\label{m2}\sup_{s\in S}\,|R\boldsymbol\mu_s|(X)=\sup_{s\in S}\,\sum_{i\in
I}\,\mu_s^i(A_i)\leqslant\sum_{i\in
I}\,a_ig_{i,\inf}^{-1}=C<\infty.\end{equation} By
Lemma~\ref{l:cl-q} with $Q=A_i$, the $\mu_s^i$, $s\in S$, are
supported by $\breve{A}_i$, $A_i$ being nearly closed. Hence,
$R\boldsymbol\mu_s^\pm$ is supported by $\breve{A}^\pm$, cf.\ (\ref{brevea}), and
therefore
\[\sup_{(x,y)\in S(R\boldsymbol\mu_s^+)\times S(R\boldsymbol\mu_s^-)}\,\kappa(x,y)\leqslant
\sup_{(x,y)\in\breve{A}^+\times\breve{A}^-}\,\kappa(x,y)<\infty\text{
\ for all\ }s\in S,\] where the latter inequality holds by
(\ref{bound}). Combining this with (\ref{m2}) establishes the
inequality
\[\kappa(R\boldsymbol\mu_s^+,R\boldsymbol\mu_s^-)\leqslant M<\infty\text{ \ for all\ }s\in S,\]
which together with (\ref{net:bound}) yields
\begin{equation}\label{net:boundpm}\sup_{s\in
S}\,\|R\boldsymbol\mu_s^\pm\|_\kappa<\infty.\end{equation}

We next observe that for every $i\in I$,
\begin{equation}\label{net:boundpmi}\sup_{s\in S}\,\|\mu_s^i\|_\kappa<\infty.\end{equation} In view of
(\ref{net:boundpm}), this will follow once we have established the
inequality
\begin{equation}\label{in:aux}\inf_{s\in S}\,\sum_{k,m\in I^\pm,\ k\ne
m}\,\kappa(\mu_s^k,\mu_s^m)>-\infty.\end{equation} Assume $X$ is
compact, for if not, then $\kappa\geqslant0$ and the left-hand side
in (\ref{in:aux}) is ${}\geqslant0$. But then the l.s.c.\ function
$\kappa$ on $X\times X$ is ${}\geqslant-c$, where $c\in(0,\infty)$, while $I$ is
finite; and hence (\ref{in:aux}) follows from (\ref{m2}) in a way
similar to that in the proof of Lemma~\ref{lemma:minusfinite}.

As seen from (\ref{m2}), the net $(\boldsymbol\mu_s)_{s\in S}$ is
vaguely bounded, and hence, by Lemma~\ref{lem:vaguecomp}, it is
relatively compact in the vague topology on $\mathfrak
M^+(X)^{\mathrm{Card}\,I}$. Thus, there is a subnet
$(\boldsymbol\mu_t)_{t\in T}$ of the net $(\boldsymbol\mu_s)_{s\in
S}$ such that for every $i\in I$,
\begin{equation}\label{i}\mu_t^i\to\mu^i\text{ \ vaguely as $t$ increases along $T$},\end{equation}
where $\mu^i\in\mathfrak M^+(X)$. It follows from
(\ref{net:boundpmi}) and (\ref{i}) by Lemmas~\ref{l:cl-q}
and~\ref{l:aux2} that $\mu^i\in\mathcal E^+_\kappa(\breve{A}_i)=\mathcal E^+_\kappa(A_i)$, and
hence $\boldsymbol\mu:=(\mu^i)_{i\in I}\in\mathfrak M^+(\mathbf A)$.

Moreover, $R\boldsymbol\mu^\pm$ is the vague limit of
$R\boldsymbol\mu_t^\pm$ as $t$ increases along $T$, which is obtained
from (\ref{i}) according to Lemma~\ref{lem:vague'}. Applying
\cite[Lemma~2.2.1(e)]{F1}, we therefore see from (\ref{net:boundpm})
that the energy of $R\boldsymbol\mu^\pm$ is finite. Since $\kappa$
is positive definite, so is
$\kappa(R\boldsymbol\mu^+,R\boldsymbol\mu^-)$ (see, e.g.,
\cite[Lemma~3.1.1]{F1}), and altogether $R\boldsymbol\mu\in\mathcal
E_\kappa(X)$. In view of (\ref{iff}), we thus have
\[\boldsymbol\mu\in\mathcal E^+_\kappa(\mathbf A).\]

Applying
Lemma~\ref{lemma:lower} to the continuous function
$g_i>0$, we also obtain from~(\ref{i})
\begin{equation}\label{g-ineq}\langle
g_i,\mu^i\rangle\leqslant\lim_{t\in T}\,\langle
g_i,\mu_t^i\rangle=a_i.\end{equation} Noting that
$\xi^i-\mu_t^i\geqslant0$ for all $i\in I_0$ and $t\in T$ as well as
that the vague limit of a net of positive (scalar) measures likewise
is positive, we finally see from (\ref{i}) that
$\mu^i\leqslant\sigma^i$ for all $i\in I$. This together with the
two preceding displays shows that, actually,
$\boldsymbol\mu\in\mathcal E_\kappa^{\boldsymbol\sigma}(\mathbf A,
\leqslant\!\mathbf a,\mathbf g)$, which establishes (\ref{inclusion1}).
\end{proof}

\begin{lemma}\label{l:cons} Let\/ {\rm(\ref{ser2})} and\/ {\rm(\ref{bound})} both
hold, and let\/ $\kappa$ be consistent. For every\/
$(\boldsymbol\mu_s)_{s\in S}\in\mathbb
M^{\boldsymbol\sigma}_{\kappa,\mathbf f}(\mathbf A,\mathbf a,\mathbf
g)$ and every\/ $\boldsymbol\mu\in(\boldsymbol\mu_s)'_{s\in S}$,
\begin{equation}\label{s-srt-conv}\lim_{s\in
S}\,\|\boldsymbol\mu_s-\boldsymbol\mu\|_{\mathcal E_\kappa^+(\mathbf
A)}=0,\end{equation}
\begin{equation}\label{sol1}-\infty<G_{\kappa,\mathbf
f}(\boldsymbol\mu)\leqslant\lim_{s\in S}\,G_{\kappa,\mathbf
f}(\boldsymbol\mu_s)=G^{\boldsymbol\sigma}_{\kappa,\mathbf
f}(\mathbf A,\mathbf a,\mathbf g)<\infty.\end{equation}
\end{lemma}

\begin{proof}We use tacitly the notation and assertions from the proof of the preceding lemma.
Being consistent, the kernel $\kappa$ possesses the property (C$_2$)
(see Section~\ref{sec:cons}). The strongly bounded net
$(R\boldsymbol\mu_t^\pm)_{t\in T}\subset\mathcal E^+_\kappa(X)$
therefore converges weakly to its vague limit
$R\boldsymbol\mu^\pm\in\mathcal E^+_\kappa(X)$, which by the
definition of weak convergence implies that $R\boldsymbol\mu_t\to R\boldsymbol\mu$ weakly as $t$ increases along $T$.
By (\ref{seminorm}), this gives
\[\|\boldsymbol\mu_t-\boldsymbol\mu\|^2_{\mathcal E^+_\kappa(\mathbf
A)}=\|R\boldsymbol\mu_t-R\boldsymbol\mu\|^2_\kappa=\lim_{t'\in
T}\,\kappa(R\boldsymbol\mu_t-R\boldsymbol\mu,R\boldsymbol\mu_t-R\boldsymbol\mu_{t'}),\]
and hence, by the Cauchy--Schwarz inequality in $\mathcal E_\kappa(X)$,
\[\|\boldsymbol\mu_t-\boldsymbol\mu\|_{\mathcal E^+_\kappa(\mathbf
A)}\leqslant\liminf_{t'\in
T}\,\|\boldsymbol\mu_t-\boldsymbol\mu_{t'}\|_{\mathcal
E^+_\kappa(\mathbf A)}\text{ \ for all\ }t\in T,\] which establishes the relation
\begin{equation*}\label{conv:str}\lim_{t\in T}\,
\|\boldsymbol\mu_t-\boldsymbol\mu\|_{\mathcal E^+_\kappa(\mathbf
A)}=0,\end{equation*} because
$\|\boldsymbol\mu_t-\boldsymbol\mu_{t'}\|_{\mathcal
E^+_\kappa(\mathbf A)}$ becomes arbitrarily small when $t,t'\in T$
are sufficiently large. Since a strong Cauchy net converges strongly
to any of its strong cluster points, we obtain (\ref{s-srt-conv}) from the last
display.

As for (\ref{sol1}), we first note that the equality and the third
inequality here are valid by (\ref{min-seq:}) and the permanent
assumption (\ref{Gconfin}), respectively. If Case~II takes place,
then the first inequality is obvious by (\ref{yusss}), while the
second inequality holds (with equality prevailing) again by (\ref{yusss}),
applied respectively to $\boldsymbol\mu_s$, $s\in S$, and
$\boldsymbol\mu$, and the subsequent use of (\ref{s-srt-conv}). Assume now
Case~I holds. Applying Lemma~\ref{lemma:lower} to $f_i\in\Psi(X)$, we see from (\ref{i}) after summation over $i\in I$ that\footnote{Note that, while proving
(\ref{fcase1}) in Case~I,
we have not used the assumption of consistency of the kernel.\label{foot-f}}
\begin{equation}\label{fcase1}-\infty<\langle\mathbf f,\boldsymbol\mu\rangle\leqslant\liminf_{t\in T}\,
\langle\mathbf f,\boldsymbol\mu_t\rangle.\end{equation}
The former inequality
here is obvious if $X$ is noncompact, while otherwise it can be
obtained from (\ref{ser2}) and (\ref{g-ineq}) in the same manner as
in the proof of Lemma~\ref{lemma:minusfinite}. Combining (\ref{fcase1}) and (\ref{s-srt-conv})
completes the proof
of~(\ref{sol1}).
\end{proof}

\subsection{Proof of Theorem~\ref{th1}}\label{th1:proof} Fix $(\boldsymbol\mu_s)_{s\in S}\in\mathbb
M^{\boldsymbol\sigma}_{\kappa,\mathbf f}(\mathbf A,\mathbf a,\mathbf
g)$ and $\boldsymbol\mu\in(\boldsymbol\mu_s)_{s\in S}'$ (such
$\boldsymbol\mu$ exists by Lemma~\ref{v-r-c}). For these $(\boldsymbol\mu_s)_{s\in S}$ and $\boldsymbol\mu$, there hold
Lemmas~\ref{v-r-c} and~\ref{l:cons} as well as the assertions
in their proofs.

We assert that $\boldsymbol\mu$ solves Problem~\ref{pr2}. We first show that, actually,
\begin{equation}\label{sol2}\boldsymbol\mu\in\mathcal E_{\kappa,\mathbf
f}^{\boldsymbol\sigma}(\mathbf A,\mathbf a,\mathbf g).\end{equation}
As seen from (\ref{inclusion1}) and (\ref{sol1}), it is enough to prove that for any given $i\in I$,
\begin{equation}\label{g-mass}\langle g_i,\mu^i\rangle=a_i.\end{equation}

It follows from Lemma~\ref{l:cl-q} with
$Q=A_i$ that the $\mu^i_s$, $s\in S$, and $\mu^i$ are carried by $A_i\cap\breve{A}_i$. There
is therefore no loss of generality in replacing each $\xi^i$, $i\in I_0$, by the extension
of $\xi^i|_{A_i\cap\breve{A}_i}$ by $0$ to all of $X$, denoted again by $\xi^i$.\footnote{As $A_i\cap\breve{A}_i$ is $\xi^i$-meas\-ur\-able, $\xi^i|_{A_i\cap\breve{A}_i}$ exists. Given $Q\subset X$, the \emph{extension\/} of $\nu\in\mathfrak M^+(Q)$ by $0$ to all of $X$ is $\tilde{\nu}\in\mathfrak M^+(X)$ determined uniquely by the relation $\tilde{\nu}(\varphi):=\langle\varphi|_Q,\nu\rangle$ for all $\varphi\in C_0(X)$.\label{extension}} Observe that for this (new) $\xi^i$, assumption (\ref{boundd}) remains valid.

Consider the exhaustion of the (closed) set $\breve{A}_i$ by the upper directed family $(K)$
of all compact subsets of $\breve{A}_i$.\footnote{A family $\mathfrak Q$
of sets $Q\subset X$ is said to be {\it upper directed\/} if for any
$Q_1,Q_2\in\mathfrak Q$ there exists $Q_3\in\mathfrak Q$ such that
$Q_1\cup Q_2\subset Q_3$.} Let $(\boldsymbol\mu_t)_{t\in T}$ be a
subnet of $(\boldsymbol\mu_s)_{s\in S}$ converging vaguely to
$\boldsymbol\mu$ (see the proof of Lemma~\ref{v-r-c}). Since the
indicator function $1_K$ of $K$ is upper semicontinuous,
Lemma~\ref{lemma:lower} with $\psi=-g_i1_{K}=-g_i|_K$ and
\cite[Lemma~1.2.2]{F1} yield
\begin{align*}a_i&\geqslant\langle g_i,\mu^i\rangle=\lim_{K\in(K)}\,\langle
g_i,\mu^i|_K\rangle=\lim_{K\in(K)}\,\langle
g_i|_K,\mu^i\rangle\geqslant\limsup_{(t,K)\in T\times(K)}\,\langle
g_i|_K,\mu_t^i\rangle\\&{}=\limsup_{(t,K)\in T\times(K)}\,\langle
g_i,\mu_t^i|_K\rangle=a_i-\liminf_{(t,K)\in
T\times(K)}\,\langle g_i,\mu_t^i|_{\breve{A}_i\setminus K}\rangle,
\end{align*}
$T\times(K)$ being the upper directed product of the upper directed
sets $T$ and $(K)$ \cite[Chapter~2, Section~3]{K}.
The first inequality here holds by (\ref{g-ineq}), while the second and third equalities follow from Lemma~\ref{up-int},
the $\mu_t^i$, $t\in T$, and $\mu^i$ being bounded.
Hence, (\ref{g-mass}) will be established once we have proven
\begin{equation}\label{g0}
\liminf_{(t,K)\in T\times(K)}\,\langle g_i,\mu_t^i|_{\breve{A}_i\setminus
K}\rangle=\liminf_{(t,K)\in T\times(K)}\,\langle g_i|_{\breve{A}_i\setminus
K},\mu_t^i\rangle=0,
\end{equation}
the former equality here being valid again according to Lemma~\ref{up-int}.

Assume first that $i\in I_0$. Since by (\ref{boundd}) and
\cite[Lemma~1.2.2]{F1}
\[\infty>\langle g_i,\xi^i\rangle=\lim_{K\in(K)}\,\langle g_i,\xi^i|_K\rangle,\]
we have
\[\lim_{K\in(K)}\,\langle g_i,\xi^i|_{\breve{A}_i\setminus K}\rangle=0.\]
When combined with
\[\langle g_i,\mu_t^i|_{\breve{A}_i\setminus K}\rangle
\leqslant\langle g_i,\xi^i|_{\breve{A}_i\setminus K}\rangle\text{ \ for
all\ }t\in T,\]
this implies the latter equality in (\ref{g0}) and, hence,
(\ref{g-mass}).

Let now $i\in I\setminus I_0$. Assuming first that $\breve{A}_i$ is
compact, we obtain (\ref{g-mass}) from (\ref{i}) in view of the continuity of $g_i$. In fact, there is $\varphi_i\in C_0(X)$ such that
$\varphi_i|_{\breve{A}_i}=g_i|_{\breve{A}_i}$. Since $\breve{A}_i^c$ is $\nu$-neg\-lig\-ible for any $\nu\in\mathfrak M^+(\breve{A}_i)$, we thus get
\begin{equation}\label{compact}a_i=\lim_{t\in T}\,\langle g_i,\mu_t^i\rangle=\lim_{t\in T}\,\langle\varphi_i,\mu_t^i\rangle=\langle\varphi_i,\mu^i\rangle=\langle g_i,\mu^i\rangle.\end{equation}

Assume next $\breve{A}_i$ is noncompact. Then, by the stated hypotheses,
\begin{equation}c_\kappa(\breve{A}_i)<\infty.\label{capfinite}\end{equation}
Since the kernel $\kappa$ is consistent, for every $Q\subset\breve{A}_i$
there exists an interior equilibrium measure $\gamma_Q$
\cite[Theorem~4.1]{F1}. Recall that, if $\Gamma_Q$ denotes the convex cone of all
$\nu\in\mathcal E_\kappa(X)$ with
$\kappa(x,\nu)\geqslant1$ n.e.\ on $Q$, then $\gamma_Q\in\Gamma_Q$,
i.e.,
\begin{equation}
\kappa(x,\gamma_Q)\geqslant1\text{ \ n.e.\ on\ }Q, \label{6}
\end{equation}
and moreover
\begin{equation}
c_\kappa(Q)=\|\gamma_Q\|_\kappa^2=\min_{\nu\in\Gamma_Q}\,\|\nu\|^2_\kappa.
\label{5}
\end{equation}

Observe that there is no loss of generality in assuming $g_i$ to
satisfy (\ref{growth}) with some $r_i\in(1,\infty)$ and
$\nu_i\in\mathcal E_\kappa(X)$. Indeed, otherwise $g_i$ must be
bounded from above by $M\in(0,\infty)$, which combined with
(\ref{6}) for $Q=\breve{A}_i$ results in (\ref{growth}) with
$\nu_i:=M^{r_i}\gamma_{\breve{A}_i}$, $r_i\in(1,\infty)$ being arbitrary.

Consider interior equilibrium measures $\gamma_{\breve{A}_i\setminus K}$ and
$\gamma_{\breve{A}_i\setminus K'}$, where $K,K'\in(K)$. Because of~(\ref{6})
and~(\ref{5}), we obtain from \cite[Lemma~4.1.1]{F1}
\[
\|\gamma_{\breve{A}_i\setminus K}-\gamma_{\breve{A}_i\setminus
K'}\|_\kappa^2\leqslant\|\gamma_{\breve{A}_i\setminus
K}\|_\kappa^2-\|\gamma_{\breve{A}_i\setminus K'}\|_\kappa^2\text{ \ whenever
\ }K\subset K'.\] As seen from (\ref{capfinite}) and (\ref{5}), the
net $(\|\gamma_{\breve{A}_i\setminus K}\|_\kappa)_{K\in(K)}$ is bounded and
decreasing, and hence it is Cauchy in $\mathbb R$. The preceding
inequality thus shows that the net $(\gamma_{\breve{A}_i\setminus
K})_{K\in(K)}$ is strong Cauchy in $\mathcal E^+_\kappa(X)$. Since,
clearly, this net converges vaguely to zero, the property (C$_1$)
implies that zero is also one of its strong limits. Hence,
\begin{equation}
\lim_{K\in(K)}\,\|\gamma_{\breve{A}_i\setminus K}\|_\kappa=0. \label{27}
\end{equation}

Write $q_i:=r_i(r_i-1)^{-1}$, where $r_i\in(1,\infty)$ is the number
involved in condition (\ref{growth}). Combining (\ref{growth}) with
(\ref{6}) shows that the inequality
\[
g_i(x)1_{\breve{A}_i\setminus K}(x)\leqslant\kappa(x,\nu_i)^{1/r_i}
\kappa(x,\gamma_{\breve{A}_i\setminus K})^{1/q_i}\]  holds n.e.\ on $\breve{A}_i$,
and hence $\mu_t^i$-a.e.\ on $X$ (see Lemma~\ref{a.e.}). Having integrated this relation with respect to
$\mu_t^i$, we then apply the H\"older and, subsequently, the
Cauchy--Schwarz inequalities to the integrals on the right. This
gives
\begin{align*}\langle
g_i1_{\breve{A}_i\setminus
K},\mu_t^i\rangle&\leqslant\Bigl[\int\kappa(x,\nu_i)\,d\mu_t^i(x)\Bigr]^{1/r_i}\,
\Bigl[\int\kappa(x,\gamma_{\breve{A}_i\setminus K})\,d\mu_t^i(x)\Bigr]^{1/q_i}\\
{}&\leqslant \|\nu_i\|_\kappa^{1/r_i}\,\|\gamma_{\breve{A}_i\setminus
K}\|_\kappa^{1/q_i}\,\|\mu_t^i\|_\kappa.\end{align*} Taking limits
here as $(t,K)$ increases along $T\times(K)$ and using
(\ref{net:boundpmi}) and (\ref{27}), we again obtain the latter equality in (\ref{g0}), and
hence~(\ref{g-mass}).

Having thus proven (\ref{sol2}), we get $G_{\kappa,\mathbf
f}(\boldsymbol\mu)\geqslant G^{\boldsymbol\sigma}_{\kappa,\mathbf
f}(\mathbf A,\mathbf a,\mathbf g)$. Since the converse inequality
holds by (\ref{sol1}), $\boldsymbol\mu$ is a solution to
Problem~\ref{pr2}, i.e., $\boldsymbol\mu\in\mathfrak
S^{\boldsymbol\sigma}_{\kappa,\mathbf f}(\mathbf A,\mathbf a,\mathbf
g)$. As $(\boldsymbol\mu_s)_{s\in S}\in\mathbb
M^{\boldsymbol\sigma}_{\kappa,\mathbf f}(\mathbf A,\mathbf a,\mathbf
g)$ and $\boldsymbol\mu\in(\boldsymbol\mu_s)'_{s\in S}$ have been
chosen arbitrarily, we obtain
\[\bigcup_{(\boldsymbol\nu_t)_{t\in
T}\in\mathbb M^{\boldsymbol\sigma}_{\kappa,\mathbf f}(\mathbf
A,\mathbf a,\mathbf g)}\,(\boldsymbol\nu_t)_{t\in
T}'\subset\mathfrak S^{\boldsymbol\sigma}_{\kappa,\mathbf f}(\mathbf
A,\mathbf a,\mathbf g).\] The converse inclusion is
obvious because the trivial net
$(\boldsymbol\lambda^{\boldsymbol\sigma}_{\mathbf A})$, where
$\boldsymbol\lambda^{\boldsymbol\sigma}_{\mathbf A}$ is any element of $\mathfrak
S^{\boldsymbol\sigma}_{\kappa,\mathbf f}(\mathbf A,\mathbf a,\mathbf
g)$, is minimizing and converges vaguely to
$\boldsymbol\lambda^{\boldsymbol\sigma}_{\mathbf A}$. Thus,
(\ref{desc-solv}) indeed holds.

Any net in $\mathfrak S^{\boldsymbol\sigma}_{\kappa,\mathbf
f}(\mathbf A,\mathbf a,\mathbf g)$ is obviously minimizing, and hence, according to (\ref{desc-solv}), any of its vague cluster points again belongs to $\mathfrak S^{\boldsymbol\sigma}_{\kappa,\mathbf f}(\mathbf
A,\mathbf a,\mathbf g)$. This
establishes the vague compactness of $\mathfrak
S^{\boldsymbol\sigma}_{\kappa,\mathbf f}(\mathbf A,\mathbf a,\mathbf
g)$. Choosing finally any $(\boldsymbol\nu_t)_{t\in T}\in\mathbb
M^{\boldsymbol\sigma}_{\kappa,\mathbf f}(\mathbf A,\mathbf a,\mathbf
g)$ and $\boldsymbol\lambda^{\boldsymbol\sigma}_{\mathbf
A}\in\mathfrak S^{\boldsymbol\sigma}_{\kappa,\mathbf f}(\mathbf
A,\mathbf a,\mathbf g)$, we arrive at (\ref{min-conv-str}) by
combining (\ref{s-srt-conv}) with Lemmas~\ref{l:unique}
and~\ref{l:fund}. The proof is complete.

\subsection{Proof of Theorem~\ref{th2}}\label{sec:pr:th2} Under the stated hypotheses, there is no loss of generality in assuming $I^+$ to be finite. Then
$\breve{A}^+$ is compact, for $\breve{A}_i$,
$i\in I$, are so. Fix $\varepsilon>0$. As seen from the property $(\infty_X)$, there exists a compact set $K\subset X$ such that
\begin{equation}\label{th2:eps}\kappa(x,y)<\frac{\varepsilon}{3C^2}\text{ \ for all\ }x\in\breve{A}^+, \ y\in K^c,\end{equation}
where $C\in(0,\infty)$ is given by (\ref{ser2}).
Since $\breve{A}^-\cap K$ is compact, it follows from (\ref{th2:disj}) and the (finite) continuity of $\kappa(x,y)$ for $x\ne y$ that
$\kappa|_{\breve{A}^+\times(\breve{A}^-\cap K)}$ is upper bounded.
This together with (\ref{th2:eps}) yields (\ref{bound}). Since (\ref{ser2}) holds by assumption, we are thus
able to use Lemma~\ref{v-r-c} as well as the assertions
established in the course of its proof.

According to Lemma~\ref{l:cl-q} with $Q=A_i$, for
every $\boldsymbol\nu\in\mathcal E^+_\kappa(\mathbf A)$ we have $\nu^i\in\mathcal E^+_\kappa(A_i\cap\breve{A}_i)$ for all $i\in I$.
There is therefore no loss of generality in replacing each $\xi^i$, $i\in I_0$, by the extension
of $\xi^i|_{A_i\cap\breve{A}_i}$ by $0$ to all of $X$ (cf.\ footnote~\ref{extension}), denoted again by $\xi^i$.
After having done this, we next
replace (again with the notation preserved) the $A_i$, $i\in I$, by the
(compact) sets $\breve{A}_i$, which again involves no loss of
generality.

By (\ref{inclusion1}), any vague cluster point $\boldsymbol\mu$ of
any $(\boldsymbol\mu_s)_{s\in S}\in\mathbb
M^{\boldsymbol\sigma}_{\kappa,\mathbf f}(\mathbf A,\mathbf a,\mathbf
g)$ belongs to $\mathcal
E_\kappa^{\boldsymbol\sigma}(\mathbf A,\leqslant\!\mathbf a,\mathbf
g)$. Choose a subsequence $\{\boldsymbol\mu_k\}_{k\in\mathbb N}$ of
$(\boldsymbol\mu_s)_{s\in S}$ that converges vaguely to
$\boldsymbol\mu$. Since $g_i$ is continuous and $A_i$ is compact,
equality prevails in (\ref{g-ineq}) (cf.\ (\ref{compact})), and hence
\begin{equation}\label{inclll}\boldsymbol\mu\in\mathcal
E_\kappa^{\boldsymbol\sigma}(\mathbf A,\mathbf a,\mathbf g).\end{equation}
Thus, $|R\boldsymbol\mu|(X)\leqslant C$, cf.\ (\ref{m2}), and
the above $K$ can be chosen so that
\begin{equation}\label{th2:nonint}|R\boldsymbol\mu^-|(\partial_XK)=0.\end{equation}

We next use the fact that the map $(\nu_1,\nu_2)\mapsto\nu_1\otimes\nu_2$
from $\mathfrak M^+(X)\times\mathfrak M^+(X)$ into
$\mathfrak M^+(X\times X)$ is vaguely continuous \cite[Chapter~3, Section~5, Exercise~5]{B2}.
Applying
Lemma~\ref{lemma:lower} to $\kappa\in\Psi(X\times X)$, we therefore obtain
\begin{equation}\label{th2:eq1}\kappa(R\boldsymbol\mu^\pm,R\boldsymbol\mu^\pm)\leqslant
\liminf_{k\to\infty}\,\kappa(R\boldsymbol\mu_k^\pm,R\boldsymbol\mu_k^\pm).\end{equation}
Furthermore,
\begin{align}\label{al1}|\kappa(R\boldsymbol\mu^+,R\boldsymbol\mu^-)-\kappa(R\boldsymbol\mu_k^+,R\boldsymbol\mu_k^-)|&\leqslant
|\kappa(R\boldsymbol\mu^+,R\boldsymbol\mu^-|_K)-\kappa(R\boldsymbol\mu_k^+,R\boldsymbol\mu_k^-|_K)|\\
\notag{}&+|\kappa(R\boldsymbol\mu^+,R\boldsymbol\mu^-|_{K^c})|+
|\kappa(R\boldsymbol\mu_k^+,R\boldsymbol\mu_k^-|_{K^c})|.\end{align}
As seen from (\ref{th2:eps}) and (\ref{m2}), each of the last two summands on the right-hand side in (\ref{al1}) is ${}<\varepsilon/3$.
Since $\kappa|_{A^+\times(A^-\cap K)}$ is continuous on the (compact) space $A^+\times(A^-\cap K)$ and $R\boldsymbol\mu_k^+\otimes(R\boldsymbol\mu_k^-|_{K})\to R\boldsymbol\mu^+\otimes(R\boldsymbol\mu^-|_{K})$ vaguely, the latter being clear from Theorem~\ref{Portmanteau} in view of (\ref{th2:nonint}), there exists $k_0\in\mathbb N$ such that for all $k\geqslant k_0$, the first summand in (\ref{al1}) is ${}<\varepsilon/3$.
Altogether,
\[\kappa(R\boldsymbol\mu^+,R\boldsymbol\mu^-)=
\lim_{k\to\infty}\,\kappa(R\boldsymbol\mu_k^+,R\boldsymbol\mu_k^-),\]
for $\varepsilon$ has been chosen arbitrarily.
Combining this with (\ref{th2:eq1}) and then substituting (\ref{vecen}) and (\ref{pos-en}) into the
inequality obtained yields
\[0\leqslant\kappa(\boldsymbol\mu,\boldsymbol\mu)\leqslant
\liminf_{k\to\infty}\,\kappa(\boldsymbol\mu_k,\boldsymbol\mu_k).\]

Since Case~I takes place, in view of footnote~\ref{foot-f} we obtain
(\ref{fcase1}), which together with the last display establishes the
relation
\[-\infty<G_{\kappa,\mathbf f}(\boldsymbol\mu)\leqslant
\liminf_{k\to\infty}\,G_{\kappa,\mathbf f}(\boldsymbol\mu_k)=
G^{\boldsymbol\sigma}_{\kappa,\mathbf f}(\mathbf A,\mathbf a,\mathbf
g)<\infty.\] The equality and the third inequality here are valid by
(\ref{min-seq:}) and the permanent assumption (\ref{Gconfin}),
respectively. In view of (\ref{inclll}), we thus actually have
$\boldsymbol\mu\in\mathcal E_{\kappa,\mathbf
f}^{\boldsymbol\sigma}(\mathbf A,\mathbf a,\mathbf g)$, and
therefore $G_{\kappa,\mathbf f}(\boldsymbol\mu)\geqslant
G^{\boldsymbol\sigma}_{\kappa,\mathbf f}(\mathbf A,\mathbf a,\mathbf
g)$. This together with the preceding display shows that
$\boldsymbol\mu$ is in fact a solution to Problem~\ref{pr2}.

It has thus been proven that any vague cluster point (which exists)
of any minimizing net (sequence) belongs to $\mathfrak S_{\kappa,\mathbf
f}^{\boldsymbol\sigma}(\mathbf A,\mathbf a,\mathbf g)$. In the same
way as it has been done at the end of the proof of
Theorem~\ref{th1}, this implies (\ref{desc-solv}) as well as the vague compactness of $\mathfrak S_{\kappa,\mathbf
f}^{\boldsymbol\sigma}(\mathbf A,\mathbf a,\mathbf g)$. The proof is complete.

\begin{remark}\label{rem:th2}Assume the conditions of footnote~\ref{foot:Part} hold. Then the corresponding version of Theorem~\ref{th2} can be proven as above (of course, with a subnet $(\boldsymbol\mu_t)_{t\in T}$ in place of a subsequence), the only difference being in the fact that Theorem~\ref{Portmanteau} may fail. Instead, choose a compact set $K$ so that $\breve{A}^+\cap \partial_XK=\varnothing$. Since this $K$ has points in common with only finitely many $\breve{A}_i$, $i\in I^+$, $(R\boldsymbol\mu^+_t|_K)_{t\in T}$ again converges vaguely to $R\boldsymbol\mu^+|_K$. Having reversed the roles of `$+$' and `$-$', we arrive at our claim.\end{remark}

\subsection{Proof of Theorem~\ref{th3}}\label{th3:proof} In the same manner as this has been done in the proof of Theorem~\ref{th2}, we see that there is no
loss of generality in replacing each $\xi^i$, $i\in I_0$, by the
extension of $\xi^i|_{A_i\cap\breve{A}_i}$ by $0$ to all of $X$
(denoted again by $\xi^i$).

We begin by showing that under the stated hypotheses the potential
$\kappa(\cdot,\nu^i)$ of any $\nu^i\in\mathfrak M^+(\breve{A}_i)$, $i\in I$,
such that $\nu^i\leqslant\sigma^i$ is continuous on $X$.
Let first $i\in I_0$. Being relatively continuous on $\breve{A}_i\supset
S(\xi^i)$ by assumption, $\kappa(\cdot,\xi^i)$ is continuous on $X$
by the regularity of the kernel. Since
$\kappa(\cdot,\nu^i)$ is l.s.c.\ and since
$\kappa(\cdot,\nu^i)=\kappa(\cdot,\xi^i)-\kappa(\cdot,\xi^i-\nu^i)$
with $\kappa(\cdot,\xi^i)$ continuous and
$\kappa(\cdot,\xi^i-\nu^i)$ l.s.c., $\kappa(\cdot,\nu^i)$ is also
upper semicontinuous, hence continuous.
Let now $i\in I\setminus
I_0$. Since $-\kappa|_{\breve{A}_i\times\breve{A}_i}$ is continuous by assumption,
$-\kappa(x,y)\geqslant-c$ for all $(x,y)\in\breve{A}_i\times\breve{A}_i$, where $c\in(0,\infty)$. Integrating this inequality with respect to the (bounded)
$\nu^i\in\mathfrak M^+(\breve{A}_i)$, we observe that $\kappa(\cdot,\nu^i)$ is relatively upper semicontinuous on
$\breve{A}_i$. Being also l.s.c.\ on $X$, it is relatively continuous on
$\breve{A}_i\supset S(\nu^i)$, and hence on all of $X$, again by the regularity of the kernel~$\kappa$.

Choose any $(\boldsymbol\mu_s)_{s\in S}\in\mathbb
M^{\boldsymbol\sigma}_{\kappa,\mathbf f}(\mathbf A,\mathbf a,\mathbf
g)$, which exists by the (permanent) assumption
(\ref{Gconfin}). By Lemma~\ref{l:cl-q} with $Q=A_i$, $i\in I$, we have
$(\mu_s^i)_{s\in S}\subset\mathcal E_\kappa^+(\breve{A}_i)$. Since $g_i$ is continuous and strictly positive,
while $\breve{A}_i$ is compact,
\[\mu_s^i(X)\leqslant a_i\bigl[\min_{x\in\breve{A}_i}\,g_i(x)\bigr]^{-1}<\infty\text{ \ for all\ }s\in S.\]
Therefore, $(\boldsymbol\mu_s)_{s\in S}$ is bounded and, hence, relatively compact in the vague topology on $\mathfrak M^+(X)^{\mathrm{Card}\,I}$ (Lemma~\ref{lem:vaguecomp}).
Fix any of its vague cluster points
$\boldsymbol\mu=(\mu^i)_{i\in I}\in\mathfrak M^+(X)^{\mathrm{Card}\,I}$, and choose a subnet $(\boldsymbol\mu_t)_{t\in T}$
of $(\boldsymbol\mu_s)_{s\in S}$ converging vaguely to $\boldsymbol\mu$.
Since $\mathfrak M^{\sigma^i}(\breve{A}_i,a_i,g_i)$ is vaguely closed (cf.\ (\ref{compact})),
\[\mu^i\in\mathfrak M^{\sigma^i}(\breve{A}_i,a_i,g_i)\text{ \ for all\ }i\in I.\]
As shown in the second paragraph of the present proof, $\kappa(\cdot,\mu^i)$ is continuous on $X$, and hence bounded on the (compact)
$\breve{A}_i$. Combined with $\mu^i(\breve{A}_i)<\infty$, this gives $\mu^i\in\mathcal E^+_\kappa(\breve{A}_i)$.
By Lemma~\ref{l:cl-q} for $Q=A_i$ and the preceding display, we thus get $\boldsymbol\mu\in\mathcal E_\kappa^{\boldsymbol\sigma}(\mathbf A,\mathbf a,\mathbf g)$ ($I$ being finite).

Furthermore, since every $\kappa(\cdot,\mu_t^i)$, $t\in T$, is likewise continuous on $X$,
\begin{align*}\lim_{t\in T}\,\lim_{t'\in T}\,\kappa(\mu_t^i,\mu_{t'}^j)&=
\lim_{t\in T}\,\lim_{t'\in T}\,\int\kappa(\cdot,\mu_t^i)\,d\mu_{t'}^j=
\lim_{t\in T}\,\int\kappa(\cdot,\mu_t^i)\,d\mu^j\\{}&=\lim_{t\in T}\,\int\kappa(\cdot,\mu^j)\,d\mu_t^i=\kappa(\mu^j,\mu^i)\text{ \ for all\ }i,j\in I.\end{align*}
Summing up these equalities, multiplied by $s_is_j$, over all $i,j\in I$ shows that $R\boldsymbol\mu_t\to R\boldsymbol\mu$ strongly in $\mathcal E_\kappa(X)$; and hence, by (\ref{seminorm}),
\[\lim_{t\in T}\,\|\boldsymbol\mu_t-\boldsymbol\mu\|_{\mathcal E^+_\kappa(\mathbf A)}=0.\]
Since a strong Cauchy net converges strongly
to any of its strong cluster points, we see that $(\boldsymbol\mu_s)_{s\in S}$ converges to $\boldsymbol\mu$ strongly in $\mathcal E^+_\kappa(\mathbf A)$, which is~(\ref{s-srt-conv}).

Applying now to $(\boldsymbol\mu_s)_{s\in S}$ and $\boldsymbol\mu$ the same arguments as in the last paragraph of the proof of Lemma~\ref{l:cons}, we arrive at (\ref{sol1}). Hence, $\boldsymbol\mu\in\mathcal E_{\kappa,\mathbf f}^{\boldsymbol\sigma}(\mathbf A,\mathbf a,\mathbf g)$.
The rest of the proof repeats word by word the last two paragraphs in the proof of Theorem~\ref{th1}.

\subsection{Proof of Corollary~\ref{cor:v:conv}}\label{cor:proof} Let the assumptions of any of
Theorems~\ref{th1}, \ref{th2}, or \ref{th3} be fulfilled. As seen from
these theorems, the class $\mathfrak S_{\kappa,\mathbf
f}^{\boldsymbol\sigma}(\mathbf A,\mathbf a,\mathbf g)$ of the
solutions to Problem~\ref{pr2} is then nonempty and given
by~(\ref{desc-solv}).

Assume moreover that the kernel $\kappa$ is strictly positive
definite, while the $A_i$, $i\in I$, are mutually essentially disjoint. By
Corollary~\ref{c:unique}, a solution to Problem~\ref{pr2} is then
unique, which implies in view of (\ref{desc-solv}) that the vague
cluster sets of the minimizing nets are identical to one another,
and all these reduce to a unique
$\boldsymbol\lambda^{\boldsymbol\sigma}_{\mathbf A}\in\mathfrak
S_{\kappa,\mathbf f}^{\boldsymbol\sigma}(\mathbf A,\mathbf a,\mathbf
g)$. Since the vague topology on $\mathfrak M^+(X)^{\mathrm{Card}\,I}$ is Hausdorff,
$\boldsymbol\lambda^{\boldsymbol\sigma}_{\mathbf A}$ must be
the vague limit of every $(\boldsymbol\nu_t)_{t\in T}\in\mathbb
M^{\boldsymbol\sigma}_{\kappa,\mathbf f}(\mathbf A,\mathbf a,\mathbf
g)$ \cite[Chapter~I, Section~9, n$^\circ$\,1]{B1}, as was to be
proven.

\section{The weighted potentials of the solutions to Problem~\ref{pr2}}

For any $\boldsymbol\nu\in\mathcal E^+_\kappa(\mathbf A)$ we denote by $W^{\boldsymbol\nu,i}_{\kappa,\mathbf f}$, $i\in I$, the $i$-component of the $\mathbf f$-weigh\-ted vector
potential $W^{\boldsymbol\nu}_{\kappa,\mathbf f}$ (cf.\ Eq.~(\ref{wpot})).

\begin{lemma}\label{aux43} $\boldsymbol\lambda\in\mathcal E^{\boldsymbol\sigma}_{\kappa,\mathbf f}(\mathbf A,\mathbf a,\mathbf g)$ solves Problem\/~{\rm\ref{pr2}} if and only if
\begin{equation}\label{aux431}\sum_{i\in I}\,\bigl\langle W^{\boldsymbol\lambda,i}_{\kappa,\mathbf f},\nu^i-\lambda^i\bigr\rangle\geqslant0\text{ \ for all\ }\boldsymbol\nu\in\mathcal E^{\boldsymbol\sigma}_{\kappa,\mathbf f}(\mathbf A,\mathbf a,\mathbf g).\end{equation}
\end{lemma}

\begin{proof}For any $\boldsymbol\mu,\boldsymbol\nu\in\mathcal E^{\boldsymbol\sigma}_{\kappa,\mathbf f}(\mathbf A,\mathbf a,\mathbf g)$ and $h\in(0,1]$, we obtain by straightforward verification
\[G_{\kappa,\mathbf f}\bigl(h\boldsymbol\nu+(1-h)\boldsymbol\mu\bigr)-G_{\kappa,\mathbf f}(\boldsymbol\mu)=2h\sum_{i\in I}\,\bigl\langle W^{\boldsymbol\mu,i}_{\kappa,\mathbf f},\nu^i-\mu^i\bigr\rangle+h^2\|\boldsymbol\nu-\boldsymbol\mu\|^2_{\mathcal E^+_\kappa(\mathbf A)}.\]
If $\boldsymbol\mu=\boldsymbol\lambda$ solves Problem~\ref{pr2}, then the left-hand (hence, the right-hand) side of this display is $\geqslant0$, for the class
$\mathcal E^{\boldsymbol\sigma}_{\kappa,\mathbf f}(\mathbf A,\mathbf a,\mathbf g)$ is convex, which leads to (\ref{aux431}) by letting $h\to0$.
Conversely, if (\ref{aux431}) holds, then the preceding formula with $\boldsymbol\mu=\boldsymbol\lambda$ and $h=1$ yields $G_{\kappa,\mathbf f}(\boldsymbol\nu)\geqslant G_{\kappa,\mathbf f}(\boldsymbol\lambda)$  for all $\boldsymbol\nu\in\mathcal E^{\boldsymbol\sigma}_{\kappa,\mathbf f}(\mathbf A,\mathbf a,\mathbf g)$, and hence $\boldsymbol\lambda\in\mathfrak S^{\boldsymbol{\sigma}}_{\kappa,\mathbf f}(\mathbf A,\mathbf a,\mathbf g)$.
\end{proof}

We next provide a description of the weighted potentials of the solutions to Problem~\ref{pr2} and single out their characteristic properties. The
permanent assumptions stated in Sections~\ref{sec:form} and~\ref{sec:perm} above are required.

\begin{theorem}\label{th:desc} Let the\/ $A_i$, $i\in I$, be nearly closed, and let\/ {\rm(\ref{ser2})}, {\rm(\ref{bound})}, and
\begin{equation}\label{as2}\sup_{x\in\breve{A}_i}\,g_i(x)<\infty\text{ \ for all\ }i\in I\end{equation}
hold.  Assume also that for every\/ $i\in I_0$,
\begin{equation}\label{as1'}\xi^i|_K\in\mathcal E^+_\kappa(X)\text{ \ for every compact\ }K\subset A^\circ_i,\end{equation}
\begin{equation}\label{as1}\xi^i(A_i\setminus A^\circ_i)=0,\end{equation}
$A^\circ_i$ being defined in\/ {\rm(\ref{nec})}. If, moreover, the\/ $f_i|_{\breve{A}_i}$, $i\in I$, are lower bounded,\footnote{If Case~I holds, then the $f_i$, $i\in I$, are necessarily lower bounded on $X$.}
then for any\/ $\boldsymbol\lambda\in\mathcal E^{\boldsymbol\sigma}_{\kappa,\mathbf f}(\mathbf A,\mathbf a,\mathbf g)$ the following two assertions are equivalent:
\begin{itemize}
\item[{\rm(i)}] $\boldsymbol\lambda\in\mathfrak S^{\boldsymbol{\sigma}}_{\kappa,\mathbf f}(\mathbf A,\mathbf a,\mathbf g)$.
\item[{\rm(ii)}] There exists\/ $(w_{\boldsymbol\lambda}^i)_{i\in I}\in\mathbb R^{{\rm Card}\,I}$ such that for all\/ $i\in I_0$,
\begin{align}
\label{b1}W^{\boldsymbol\lambda,i}_{\kappa,\mathbf f}&\geqslant w_{\boldsymbol\lambda}^ig_i\text{ \ $(\xi^i-\lambda^i)$-a.e.},\\
\label{b2}W^{\boldsymbol\lambda,i}_{\kappa,\mathbf f}&\leqslant w_{\boldsymbol\lambda}^ig_i\text{ \ $\lambda^i$-a.e.},
\end{align}
while for all\/ $i\in I\setminus I_0$,
\begin{align}
\label{b3}W^{\boldsymbol\lambda,i}_{\kappa,\mathbf f}&\geqslant w_{\boldsymbol\lambda}^ig_i\text{ \ n.e. on\ }A_i,\\
\label{b4}W^{\boldsymbol\lambda,i}_{\kappa,\mathbf f}&=w_{\boldsymbol\lambda}^ig_i\text{ \ $\lambda^i$-a.e.}
\end{align}
\end{itemize}
\end{theorem}

\begin{proof} As seen from (\ref{as1'}), the $\xi^i$, $i\in I_0$, are $c_\kappa$-ab\-sol\-ut\-ely continuous,
which will be permanently used in the proof.\footnote{As in \cite[p.~134]{L}, we call $\mu\in\mathfrak M(X)$
\emph{$c_\kappa$-ab\-sol\-ut\-ely continuous\/} if $\mu(K)=0$ for every
compact $K\subset X$ with $c_\kappa(K)=0$. Then $|\mu|_*(Q)=0$ for any $Q\subset X$ with $c_\kappa(Q)=0$. Every $\mu\in\mathcal
E_\kappa(X)$ is $c_\kappa$-ab\-sol\-ut\-ely continuous; but not conversely \cite[pp.~134--135]{L}.}
There is therefore no loss of generality in replacing each $\xi^i$, $i\in I_0$, by the
extension of $\xi^i|_{A_i\cap\breve{A}_i}$ by $0$ to all of $X$
(denoted again by $\xi^i$). After having done this, we next
replace (again with the notation preserved) the $A_i$, $i\in I$, by the
(closed) sets $\breve{A}_i$, which also involves no loss of
generality. Note that then (\ref{as1'}) and (\ref{as1}) remain valid.

For every $\boldsymbol\nu=(\nu^\ell)_{\ell\in I}\in\mathcal E^{\boldsymbol\sigma}_{\kappa,\mathbf f}(\mathbf A,\mathbf a,\mathbf
g)$ and every $i\in I$, write $\boldsymbol\nu_i:=(\nu_i^\ell)_{\ell\in I}$, where $\nu_i^\ell:=\nu^\ell$ for all $\ell\ne i$ and $\nu_i^i=0$; then $\boldsymbol\nu_i\in\mathcal E^+_{\kappa,\mathbf f}(\mathbf A)$. According to (\ref{vectorpot}) and (\ref{Rpot}), $\kappa^i_{\boldsymbol\nu_i}$ is given by
\[\kappa^i_{\boldsymbol\nu_i}(x)=s_i\sum_{\ell\in I, \ \ell\ne i}\,s_\ell\kappa(x,\nu^\ell)=s_i\kappa(x,R\boldsymbol\nu_i),\]
and it is well defined and finite n.e.\ (Corollary~\ref{relation}).

Furthermore, under the stated assumptions, \emph{$\kappa^i_{\boldsymbol\nu_i}$ is lower bounded on\/} $A_i$.
In fact, in the same manner as in the proof of Lemma~\ref{v-r-c} we see from (\ref{ser2}) that
$|R\boldsymbol\nu_i|(X)\leqslant C$, where $C\in(0,\infty)$. This together with (\ref{bound}) implies that
$\kappa(\cdot,R\boldsymbol\nu_i^-)$, resp.\ $\kappa(\cdot,R\boldsymbol\nu_i^+)$, is upper bounded on $A^+$, resp.\ on $A^-$, which in view of the preceding display yields our claim.

Fix $\boldsymbol\lambda\in\mathcal E^{\boldsymbol\sigma}_{\kappa,\mathbf f}(\mathbf A,\mathbf a,\mathbf g)$. Having written $\tilde{f}_i:=f_i+\kappa^i_{\boldsymbol\lambda_i}$, define the function
\begin{equation}\label{sW}W_{\kappa,\tilde{f}_i}^{\lambda^i}:=\kappa(\cdot,\lambda^i)+\tilde{f}_i=\kappa(\cdot,\lambda^i)+f_i+\kappa^i_{\boldsymbol\lambda_i}.\end{equation}
Comparing this with (\ref{vectorpot}) and (\ref{wpot}), we get
\begin{equation}\label{ww}W_{\kappa,\tilde{f}_i}^{\lambda^i}=W_{\kappa,\mathbf f}^{\boldsymbol\lambda,i}\text{ \ for all\ }i\in I.\end{equation}
Note that $W_{\kappa,\tilde{f}_i}^{\lambda^i}$ is finite n.e.\ on $A^\circ_i$ and lower bounded on $A_i$, because this is the case for each of the summands $\kappa(\cdot,\lambda^i)$, $f_i$, and $\kappa^i_{\boldsymbol\lambda_i}$.

To establish the equivalence of (i) and (ii), suppose first that (i) holds, i.e., $\boldsymbol\lambda\in\mathcal E^{\boldsymbol\sigma}_{\kappa,\mathbf f}(\mathbf A,\mathbf a,\mathbf g)$ solves Problem~\ref{pr2}. Fix $i\in I$.
By (\ref{vectoren}) and (\ref{wen}), we get for any
$\boldsymbol\nu\in\mathcal E^{\boldsymbol\sigma}_{\kappa,\mathbf f}(\mathbf A,\mathbf a,\mathbf g)$ with the additional property that $\boldsymbol\nu_i=\boldsymbol\lambda_i$ (in particular, for $\boldsymbol\nu=\boldsymbol\lambda$)
\[G_{\kappa,\mathbf f}(\boldsymbol\nu)=G_{\kappa,\mathbf f}(\boldsymbol\lambda_i)+G_{\kappa,\tilde{f}_i}(\nu^i).\]
Combined with $G_{\kappa,\mathbf f}(\boldsymbol\nu)\geqslant G_{\kappa,\mathbf f}(\boldsymbol\lambda)$, this yields $G_{\kappa,\tilde{f}_i}(\nu^i)\geqslant G_{\kappa,\tilde{f}_i}(\lambda^i)$, and hence $\lambda^i$ minimizes $G_{\kappa,\tilde{f}_i}(\nu)$, where $\nu$ ranges over the class $\mathcal E^{\sigma^i}_{\kappa,\tilde{f}_i}(A_i,a_i,g_i)$.
This enables us to show that there exists $w_{\lambda^i}\in\mathbb R$ such that
\begin{align}\label{sing1}W_{\kappa,\tilde{f}_i}^{\lambda^i}&\geqslant w_{\lambda^i}g_i\text{ \ $(\xi^i-\lambda^i)$-a.e.},\\
\label{sing2}W_{\kappa,\tilde{f}_i}^{\lambda^i}&\leqslant w_{\lambda^i}g_i\text{ \ $\lambda^i$-a.e.}\end{align}
whenever $i\in I_0$, while otherwise (for $i\in I\setminus I_0$)
\begin{align}\label{sing3}W_{\kappa,\tilde{f}_i}^{\lambda^i}&\geqslant w_{\lambda^i}g_i\text{ \ n.e. on\ }A_i,\\
\label{sing4}W_{\kappa,\tilde{f}_i}^{\lambda^i}&=w_{\lambda^i}g_i\text{ \ $\lambda^i$-a.e.}
\end{align}

To this end, write for any $w\in\mathbb R$
\begin{align*}A_i^+(w)&:=\bigl\{x\in A_i:\ W_{\kappa,\tilde{f}_i}^{\lambda^i}(x)>wg_i(x)\bigr\},\\
A_i^-(w)&:=\bigl\{x\in A_i:\ W_{\kappa,\tilde{f}_i}^{\lambda^i}(x)<wg_i(x)\bigr\},\end{align*}
and assume first that $i\in I_0$. Then (\ref{sing1}) holds with
\[w_{\lambda^i}:=L_i:=\sup\,\bigl\{t\in\mathbb R: \ W_{\kappa,\tilde{f}_i}^{\lambda^i}\geqslant tg_i\text{ \ $(\xi^i-\lambda^i)$-a.e.}\bigr\}.\]
In turn, (\ref{sing1}) with $w_{\lambda^i}=L_i$ yields $L_i<\infty$, because
\[\widetilde{W}_{\kappa,\tilde{f}_i}^{\lambda^i}:=W_{\kappa,\tilde{f}_i}^{\lambda^i}/g_i<\infty\] holds n.e.\ on $A^\circ_i$,
hence $(\xi^i-\lambda^i)$-a.e.\ on $A^\circ_i$, for $\xi^i$ and $\lambda^i$ are both $c_\kappa$-ab\-sol\-ut\-ely continuous, and finally $(\xi^i-\lambda^i)$-a.e.\ on $A_i$ by (\ref{as1}). Also, $L_i>-\infty$ since, in consequence of (\ref{as2}),
$\widetilde{W}_{\kappa,\tilde{f}_i}^{\lambda^i}$ along with $W_{\kappa,\tilde{f}_i}^{\lambda^i}$ is lower bounded on~$A_i$.

We next proceed by establishing (\ref{sing2}) with $w_{\lambda^i}=L_i$. Assume, on the contrary, that this fails to hold.
Since $\widetilde{W}_{\kappa,\tilde{f}_i}^{\lambda^i}$ is $\lambda^i$-meas\-ur\-able, one can choose $w_i\in(L_i,\infty)$ so that
$\lambda^i(A_i^+(w_i))>0$. At the same time, as $w_i>L_i$, it follows from the definition of $L_i$ that
$(\xi^i-\lambda^i)(A_i^-(w_i))>0$. Therefore, there exist compact sets $K_1\subset A_i^+(w_i)$ and $K_2\subset A_i^-(w_i)$ such that
\begin{equation}\label{c01}0<\langle g_i,\lambda^i|_{K_1}\rangle<\langle g_i,(\xi^i-\lambda^i)|_{K_2}\rangle.\end{equation}

Write $\tau^i:=(\xi^i-\lambda^i)|_{K_2}$; then $\kappa(\tau^i,\tau^i)<\infty$ by (\ref{as1'}). Since $\langle W_{\kappa,\tilde{f}_i}^{\lambda^i},\tau^i\rangle\leqslant\langle w_ig_i,\tau^i\rangle<\infty$, in view of (\ref{sW})
we get $\langle\tilde{f}_i,\tau^i\rangle<\infty$. Define
\[\theta^i:=\lambda^i-\lambda^i|_{K_1}+c_i\tau^i,\text{ \ where \ }c_i:=\langle g_i,\lambda^i|_{K_1}\rangle/\langle g_i,\tau^i\rangle.\]
Observing from (\ref{c01}) that $c_i\in(0,1)$, we obtain by straightforward verification $\langle g_i,\theta^i\rangle=a_i$ and also $\theta^i\leqslant\xi^i$. Hence, $\theta^i\in\mathcal E^{\xi^i}_{\kappa,\tilde{f}_i}(A_i,a_i,g_i)$. But
\begin{align*}
\langle W_{\kappa,\tilde{f}_i}^{\lambda^i},\theta^i-\lambda^i\rangle&=\langle
W_{\kappa,\tilde{f}_i}^{\lambda^i}- w_ig_i,\theta^i-\lambda^i\rangle\\&{}=-\langle
W_{\kappa,\tilde{f}_i}^{\lambda^i}- w_ig_i,\lambda^i|_{K_1}\rangle+c_i\langle
W_{\kappa,\tilde{f}_i}^{\lambda^i}- w_ig_i,\tau^i\rangle<0,\end{align*}
which is impossible in view of the scalar version of Lemma~\ref{aux43} (with $I=\{i\}$). The contradiction obtained establishes (\ref{sing2}).

Let now $i\in I\setminus I_0$. Since $\lambda^i$ minimizes $G_{\kappa,\tilde{f}_i}(\nu)$ among $\nu\in\mathcal E^+_{\kappa,\tilde{f}_i}(A_i,a_i,g_i)$,  it follows from \cite[Theorem~7.1]{ZPot2} that (\ref{sing3}) and (\ref{sing4}) hold with
\[w_{\lambda^i}:=\langle W_{\kappa,\tilde{f}_i}^{\lambda^i},\lambda^i\rangle/a_i\in\mathbb R.\]

Substituting (\ref{sing1}), (\ref{sing2}), (\ref{sing3}), and (\ref{sing4}) into (\ref{ww}) establishes (\ref{b1}), (\ref{b2}),
(\ref{b3}),  and (\ref{b4}) with $w_{\boldsymbol\lambda}^i:=w_{\lambda^i}$, $i\in I$. Hence, (i)$\Rightarrow$(ii).

To complete the proof, suppose finally that (ii) holds. By (\ref{ww}), for every $i\in I_0$, resp.\ $i\in I\setminus I_0$, (\ref{sing1}) and (\ref{sing2}), resp.\ (\ref{sing3}) and (\ref{sing4}), are then fulfilled with $w_{\lambda^i}:=w_{\boldsymbol\lambda}^i$ and $\tilde{f}_i:=f_i+\kappa^i_{\boldsymbol\lambda_i}$. We observe from (\ref{sing1}) and (\ref{sing2}) that
$\lambda^i(A_i^+(w_{\lambda^i}))=0$ and $(\xi^i-\lambda^i)(A_i^-(w_{\lambda^i}))=0$ for all $i\in I_0$.
Having fixed $\boldsymbol\nu\in\mathcal E^{\boldsymbol\sigma}_{\kappa,\mathbf f}(\mathbf A,\mathbf a,\mathbf g)$, we therefore get for all $i\in I_0$,
\begin{align}\label{last1}&\langle W^{\boldsymbol\lambda,i}_{\kappa,\mathbf f},\nu^i-\lambda^i\rangle=\langle
W_{\kappa,\tilde{f}_i}^{\lambda^i}-w_{\lambda^i}g_i,\nu^i-\lambda^i\rangle\\
\notag&=\langle W_{\kappa,\tilde{f}_i}^{\lambda^i}-w_{\lambda^i}g_i,\nu^i|_{A_i^+(w_{\lambda^i})}\rangle+\langle
W_{\kappa,\tilde{f}_i}^{\lambda^i}-w_{\lambda^i}g_i,(\nu^i-\xi^i)|_{A_i^-(w_{\lambda^i})}\rangle\geqslant0.\end{align}
Furthermore, it follows from (\ref{sing3}) and (\ref{sing4}) that
\[\lambda^i(A_i^+(w_{\lambda^i}))=\lambda^i(A_i^-(w_{\lambda^i}))=\nu^i(A_i^-(w_{\lambda^i}))=0\text{ \ for all\ }i\in I\setminus I_0,\]
$\nu^i$ being $c_\kappa$-absolutely continuous. Hence, for all $i\in I\setminus I_0$,
\begin{equation}\label{last2}\langle W^{\boldsymbol\lambda,i}_{\kappa,\mathbf f},\nu^i-\lambda^i\rangle=\langle
W_{\kappa,\tilde{f}_i}^{\lambda^i}-w_{\lambda^i}g_i,\nu^i-\lambda^i\rangle=\langle W_{\kappa,\tilde{f}_i}^{\lambda^i}-w_{\lambda^i}g_i,\nu^i|_{A_i^+(w_{\lambda^i})}\rangle\geqslant0.\end{equation}
Summing up the inequalities in (\ref{last1}) and (\ref{last2}) over all $i\in I$, we see from Lemma~\ref{aux43} in view of the arbitrary choice of $\boldsymbol\nu\in\mathcal E^{\boldsymbol\sigma}_{\kappa,\mathbf f}(\mathbf A,\mathbf a,\mathbf g)$ that $\boldsymbol\lambda$ is indeed a solution to Problem~\ref{pr2}.\end{proof}

\begin{corollary}\label{desc:cont}Under the hypotheses of Theorem\/~{\rm\ref{th:desc}}, assume moreover that\/ $\kappa$ is continuous on\/ $\breve{A}^+\times\breve{A}^-$ and Case\/~{\rm I} holds. Then\/ {\rm(\ref{b2})} and\/ {\rm(\ref{b4})} in\/ {\rm(ii)} are equivalent to the following apparently stronger relations:
\begin{align*}
W^{\boldsymbol\lambda,i}_{\kappa,\mathbf f}(x)&\leqslant w_{\boldsymbol\lambda}^ig_i(x)\text{ \ for all\ }x\in S(\lambda^i)\text{\ and all\ }i\in I_0,\\
W^{\boldsymbol\lambda,i}_{\kappa,\mathbf f}(x)&=w_{\boldsymbol\lambda}^ig_i(x)\text{ \ for nearly all\ }x\in S(\lambda^i)\text{\ and all\ }i\in I\setminus I_0.
\end{align*}\end{corollary}

\begin{proof}This will follow once we have proven that for any $i\in I$, $W^{\boldsymbol\lambda,i}_{\kappa,\mathbf f}|_{\breve{A}_i}$ is l.s.c., which in turn holds if
$\kappa(\cdot,R\boldsymbol\lambda^-)|_{\breve{A}^+}$ and $\kappa(\cdot,R\boldsymbol\lambda^+)|_{\breve{A}^-}$ are shown to be continuous. To establish the latter, write
\begin{equation}\label{kappastar}
\kappa^*(x,y):=-\kappa(x,y)+\sup_{(x',y')\in\breve{A}^+\times\breve{A}^-}\,\kappa(x',y'), \ (x,y)\in\breve{A}^+\times\breve{A}^-.
\end{equation}
Under the stated assumptions, $\kappa^*$ is nonnegative and
continuous, and hence
\[\kappa^*(x,R\boldsymbol\lambda^-):=\int\kappa^*(x,y)\,dR\boldsymbol\lambda^-(y), \ x\in
\breve{A}^+,\]  is l.s.c. In the same manner as in the proof of Lemma~\ref{v-r-c}, we observe from (\ref{ser2}) that $|R\boldsymbol\lambda|(X)\leqslant C<\infty$. Integrating (\ref{kappastar}) with respect to $R\boldsymbol\lambda^-$, we therefore see that $\kappa^*(x,R\boldsymbol\lambda^-)$,
$x\in\breve{A}^+$, coincides up to a finite summand with the
restriction of $-\kappa(x,R\boldsymbol\lambda^-)$ to $\breve{A}^+$. It follows that $\kappa(x,R\boldsymbol\lambda^-)|_{\breve{A}^+}$ must be upper semicontinuous. Being also l.s.c., $\kappa(x,R\boldsymbol\lambda^-)|_{\breve{A}^+}$ is actually continuous as desired. The same holds with the indices + and - reversed.
\end{proof}

{\sl Acknowledgement.} I thank Bent Fuglede for many fruitful discussions on the topic of the paper.

\end{document}